\UseRawInputEncoding  
\documentclass[twoside, 14pt]{article}
\usepackage{amsmath,amssymb,amsthm,mathrsfs,dsfont}
\usepackage[margin=2.0cm]{geometry} 
\usepackage{lipsum}
\usepackage{titlesec,hyperref}
\usepackage{fancyhdr}
\usepackage{color}

\fancypagestyle{plain}
{
\fancyhf{}

}

\titleformat{\subsection}{\it}{\thesubsection.\enspace}{1.5pt}{}
\titleformat{\subsubsection}{\it}{\thesubsubsection.\enspace}{1.5pt}{}

\newtheorem{theo}{Theorem}[section]
\newtheorem{lemm}[theo]{Lemma}

\newtheorem{prop}[theo]{Proposition}
\newtheorem{rema}[theo]{Remark}
\numberwithin{equation}{section}
\def\bel{\begin{equation}\label}

\def\eeq{\end{equation}}

\newcommand{\na}{{\nabla}}
\newcommand{\pa}{{\partial}}

\newcommand{\beq}{\begin{equation}}
\newcommand{\beno}{\begin{equation*}}
\newcommand{\eeno}{\end{equation*}}

\let\pa=\partial

\let\d=\delta

\let\lam=\lambda

\let\f=\frac

\let\tri=\triangle
\let\ep=\epsilon
\def\na{\nabla}
\def\th{\theta}
\def\dive{\mathop{\rm div}\nolimits}

\begin{document}
\title{Optimal decay of full compressible Navier-Stokes equations with potential force \hspace{-4mm}}
\author{Jincheng Gao$^\dag$ \quad Minling Li$^\ddag$ \quad Zheng-an Yao$^\sharp$ \\[10pt]
\small {School of Mathematics, Sun Yat-Sen University,}\\
\small {510275, Guangzhou, P. R. China}\\[5pt]
}

\footnotetext{Email: \it $^\dag$gaojch5@mail.sysu.edu.cn,
\it $^\ddag$limling3@mail2.sysu.edu.cn,
\it $^\sharp$mcsyao@mail.sysu.edu.cn}
\date{}

\maketitle

\begin{abstract}
In this paper, we aim to investigate the optimal decay rate for the higher order spatial derivative of global solution to the full compressible Navier-Stokes (CNS) equations with potential force in $\mathbb{R}^3$.
We establish the optimal decay rate of the solution itself and its spatial derivatives (including the highest order spatial derivative) for global small solution of the full CNS equations with potential force.
With the presence of potential force in the considered full CNS equations, the difficulty in the analysis comes from the appearance of non-trivial ststionary solutions.
These decay rates are really optimal in the sense that it coincides with the rate of the solution of the linerized system.
In addition, the proof is accomplished by virtue of time weighted energy estimate, spectral analysis, and high-low frequency decomposition.

\vspace*{5pt}

\vspace*{5pt}
\end{abstract}


\section{Introduction}
It is well-known that full compressible Navier-Stokes (CNS) equations can be used to describe the motion of compressible viscous and heat-conductive fluids.
In this paper, we are concerned with the optimal decay rate
of global small solution to the Cauchy problem for the following full CNS equations with external force in $\mathbb{R}^3$:
\beq\label{ns3}
\left\{\begin{array}{lr}
	\rho_t +\dive(\rho u)=0,\quad (t,x)\in \mathbb{R}_{+}\times \mathbb{R}^3,\\
	\rho(u_t+u\cdot \na u)+ \na p(\rho,\th)-\mu\tri u-(\mu+\lam)\na\dive u +\rho\na \phi= 0,\quad (t,x)\in \mathbb{R}_{+}\times \mathbb{R}^3,\\
	c_{\nu}	\rho(\th_t+u\cdot \na \th)+ \th p_{\th}(\rho,\th)\dive u-\kappa\tri\th-\Psi(u)= 0,\quad (t,x)\in \mathbb{R}_{+}\times \mathbb{R}^3,
\end{array}\right.
\eeq
where $\rho$, $u$, $\th$ and $p$ represent the density, velocity,  temperature and pressure, respectively.
And $-\na\phi$ is the time independent potential force. The constant viscosity coefficients $\mu$ and $\lambda$ satisfy the following physical conditions:
\[\mu>0,\quad 2\mu+3\lam\geq0.\]
And $\kappa>0$ is the coefficient of heat conduction, $c_{\nu}>0$ is the specific heat at constant volume.
The initial data
\beq\label{initial-boundary}
(\rho,u,\th)|_{t=0}=(\rho_0,u_0,\th_0)(x)\rightarrow (\rho_{\infty},0,\th_\infty),\quad \text{as}~~|x|\rightarrow \infty,
\eeq
where $\rho_{\infty}$ and $\th_{\infty}$ are two positive constants.
The pressure $p(\rho,\th)$ here is assumed to be a smooth function in a neighborhood of $(\rho_\infty,\th_\infty)$ satisfying $p_\rho(\rho,\th)>0$ and $p_\th(\rho,\th)>0$. And $\Psi(u)$ is the classical dissipation function satisfying
$$\Psi(u)=\f{\mu}{2}(\na u+\na^T u)^2+\lam(\dive u)^2.$$

In this paper, we will investigate the optimal convergence rates in time to the stationary solution of the Cauchy problem \eqref{ns3}-\eqref{initial-boundary}.
It is noted that the large time behavior of the solution is an important topic in the research of the fluid dynamics for achieving the goal of the computation, one may refer to \cite{ene1996,goubet1996}.
And the stationary solution $(\rho^*(x),u^*(x),\th^*(x))$ for
the full CNS equations \eqref{ns3} is given by $(\rho^*(x),0,\th_\infty)$ satisfying
\beq\label{p}
\int_{\rho_{\infty}}^{\rho^*(x)}\f{p_{\rho}(s,\theta_\infty)}{s}ds+\phi(x)=0.
\eeq
The derivation for the stationary solution was given by Matsumura and Nishida in \cite{mat1983}, so we omit here.

We will give an overview some known results on the mathematical analysis on existence, stability, large time behavior, and convergence rates of solutions to the CNS equations.

\textbf{Some Results without External Force.}
There are huge literatures on the well-posedness and large time behavior of solutions to the CNS equations without external force.
It is well-known that the local existence and uniqueness of classical solutions were obtained in \cite{{Serrin-1959},{Nash-1962}} in the absence of vacuum.
For the case that the initial density may vanish in open sets, the similar results may refer to \cite{{Cho-Choe-Kim-2004},{Choe-Kim-2003},{Cho-Kim-2006},{Li-Liang-2014},{Salvi-Straskraba-1993}}.
The first global well-posedness result goes back to Matsumura and Nishida \cite{Matsumura-Nishida-1980}. 
It is noted that this famed result requires that the initial data is closed to a non-vacuum equilibrium in some Sobolev space $H^s$.
In other words, the solution is a small oscillation around a uniform non-vacuum state, which guarantees that the density is strictly away from vacuum.
In the framework of general data, this is a challenging problem, due to the difficulty in the analysis comes from the possible appearance of vacuum.
It is indicated in \cite{Li2019,rozanova2008blow,Xin1998,Xin2013} that the strong (or smooth) solution for the CNS equations will blow up in finite time.
Then some blow-up criteria of strong solutions were given in \cite{{Huang2011blowcriter1},{Sun-Wang-Zhang-2011},{Huang2011blowcriter2},{Wen2013}} and the references therein.
When the vaccum is taken into account, the global existence and uniqueness of strong solution for the full CNS equations in $\mathbb{R}^3$ was achieved by Huang et al.\cite{Huang-Li2018} for small initial energy.
Similar result was obtained for CNS equations in \cite{Huang-Li-Xin-2012,Li2020,Li-Xin-2019-pde,Wen2017}.
It is worth noting that all results on the global dynamics about the stability and large time behavior are restricted to the regime that the solutions are close to the equilibrium.
For large data, it is well-known that Feireisl constructed the so-called variational solutions for specific pressure laws excluding the perfect gas in \cite{Feireisl2004}, when the temperature equation is satisfied only as an inequality.
Then for a special form of the viscosity coefficients depending on the density, Bresch and Desjardins \cite{Bresch2007} obtained the existence of global weak solutions by making use of a new entropy inequality (BD-entropy), which was proposed in \cite{Bresch2004}, and the construction scheme of approximate solutions in \cite{Bresch2006}.
With the aid of BD-entropy, there are  Some other related results with respect to the global well-posedness theory of weak solution, one may refer to \cite{GuoZH2008,Li-Xin-2015,Vasseur2016}.
Recently, He et al. \cite{he-huang-wang} investigated global stability of large solution and established the decay rate for the global solution as it tends to the constant equilibrium state under the assumption that $\sup_{t\in \mathbb R^+}\|\rho(t, \cdot)\|_{C^\alpha}\le  M$ for some $0<\alpha<1$.
Later, the optimal decay rate for this class of global large solution itself and its derivatives was investigated, one may refer to \cite{{gao-wei-yao-D},{gao-wei-yao-NA},{gao-wei-yao-pre}}.

In the past and recent years, important progress has been made in the large time behavior of the solutions to the CNS system in the near equilibrium regime.
The decay result was first achieved by Matsumura and Nishida in \cite{nishida2} for the optimal $L^2$ decay rate, and later by Ponce in \cite{ponce} for the optimal $L^p(p\ge 2)$ decay rate.
Hoff \cite{hoff-zumbrun}, Liu and Wang \cite{liu-wang} obtained the optimal $L^p$ ($1\le p\le\infty$) decay rates in $\mathbb R^n(n\ge 2)$ by virtue of the good properties of Green function with the small initial perturbation, which bounded in $H^s\cap L^1$ with the integer $s\ge[n/2]+3$.
Furthermore, developed by Schonbek \cite{Schonbek1985}, Gao et al. \cite{gao2016} established optimal decay rate for the higher-order spatial derivative of global small solution by using the Fourier splitting method.
The approach to proving all these decay results mentioned above relies heavily on the analysis of the linearization of the system.
From another point of view, under the assumption that the initial perturbation is bounded in $\dot H^{-s}(s\in [0, \frac32))$, Guo and Wang \cite{guo2012} applied pure energy method to build the optimal decay rate of the solution and its spatial derivatives of the CNS system under the $H^N(N\ge 3)-$framework.
However, the decay rate for the highest order spatial derivative of global solution obtained in articles mentioned above is still not optimal.
Recently, this tricky problem is addressed simultaneously in a series of articles \cite{{chen2021},{Wang2020},{wu2020}} by using the spectrum analysis of the linearized system,
and \cite{gao2021} by combining the energy estimates with the interpolation between negative and positive Sobolev spaces.

\textbf{Some Results with Potential Force.}
When there is an external potential force, there are also some results on the convergence rate for solutions to the CNS equations.
For potential force, one has to face a tricky problem of the appearance of non-trivial stationary solutions.
However, some seminal results on the existence and large behavior theory were still achieved.
The global solutions was first obtained by Matsumura and Nishida in \cite{mat1983} as initial data is closed to the steady state $(\rho^*(x),0,\theta_\infty)$ in the Sobolev space $H^3$.
In addition, they also proved that the global solution converges to the stationary state as time tends to infinity.
The background profile is non-trivial due to the effect of the external force, thus, unlike the problems without potential force, the analysis on the convergence rates is more delicate and difficult.
The first work concerning the explicit decay estimate for solution was done by Deckelnick in \cite{Deckelnick1992}.
More precisely, he considered the isentropic case and showed that
\beno
\sup_{x\in\mathbb{R}^3}|(\rho(t,x)-\rho^*(x),u(t,x))|\lesssim(1+t)^{-\f14}.
\eeno
This result was then improved by Shibata and Tanaka in \cite{Shibata2003,Shibata2007} for more general external forces to $(1+t)^{-\f12+\kappa}$ for any small positive constant $\kappa$ when the initial perturbation belongs to $H^3\cap L^{\f65}$, and later by Duan et al. in \cite{duan2007} for $L^p-L^q$ convergence rates when the initial perturbation is also bounded in $L^p$ with $1\leq p<\f65$.
However, in \cite{duan2007}, the decay estimates of the higher order spatial derivatives of the solution were obtained the same as that of the first order one.
Recently, Gao et al. \cite{gao2021} improved this result under $H^N-$framework ($N\geq 3$). Specifically, they established the optimal decay rate of $k-$th $(k=0,1,\cdots,N)$ order spatial derivative (including the highest order spatial derivative) of the solution.

Most of the above results are for isentropic fluids without taking the effect of heat-conductivity into account.
In many physical situations, the heat-conductivity is an important driving force to motions of fluids.
When the heat-conductivity is taken into account for compressible flows, Duan et al. \cite{duan-ukai-yang-zhao2007} obtained the optimal decay rates in $H^3-$framework for system \eqref{ns3}-\eqref{initial-boundary} when the initial perturbation is also bounded in $L^1$.
To be more specific, they established the following decay estimates:
\beq\label{highdecayL1}
\begin{split}
	\|(\rho-\rho^*,u,\theta-\theta_\infty)(t)\|_{L^p}\lesssim (1+t)^{-\frac32(1-\f1p)},\quad 2\leq p\leq6,\quad
	\|\na(\rho-\rho^*,u,\theta-\theta_\infty)(t)\|_{H^2}
	\lesssim(1+t)^{-\f54}.
\end{split}
\eeq
For the case that the initial perturbation belongs to $H^2\cap L^q (1\leq q\leq 2)$, Wang \cite{Wang2017} established the following optimal time decay rates for all $t\geq0$,
\beq\label{highdecayH2}
\begin{split}
	\|(\rho-\rho^*,u,\theta-\theta_\infty)(t)\|_{L^p}\lesssim (1+t)^{-\frac32(\f1q-\f1p)},~ 2\leq p\leq6,~
	\|\na(\rho-\rho^*,u,\theta-\theta_\infty)(t)\|_{H^1}\lesssim (1+t)^{-\frac32(\f1q-\f12)-\f12}.
\end{split}
\eeq
There are some results concerning the decay estimate for the CNS equations with potential force, as observed in \cite{{Ukai2006},{duan2007},{Okita2014},{Li2011},{Matsumura1992},{Matsumura2001}}.
Obviously, the decay rates of the higher order spatial derivatives in either \eqref{highdecayL1} or \eqref{highdecayH2} are still not optimal.
Thus it is of interest to investigate the optimal decay rate for the higher order derivative of global solutions to system \eqref{ns3}-\eqref{initial-boundary} in three dimensions.
\textbf{\textit{Based on the decay result \eqref{highdecayL1}, our main purpose in this paper is to establish the optimal decay rate for the $k-th$ $(k=2,3)$ order spatial derivative of the solution to the full CNS equations with potential force.}}

Now, the optimal convergence rates for solutions and its spatial derivatives of Cauchy problem \eqref{ns3}-\eqref{initial-boundary} in $L^2-$norm can be obtained and stated as follows:

\begin{theo}\label{them3}
	Let $(\rho^*(x),0,\theta_\infty)$ be the stationary solution of initial value problem \eqref{ns3}--\eqref{initial-boundary}, if $(\rho_0-\rho^*,u_0,\theta_0-\theta_\infty)\in H^3$, there exists a constant $\delta$ such that the potential function $\phi(x)$ satisfies
	\beq\label{phik}
	\sum_{k=0}^{4}\|(1+|x|)^{k}\na^k\phi\|_{L^2\cap L^\infty}\leq \delta,
	\eeq
	and the initial perturbation statisfies
	\beq\label{initial-H2}
	\|(\rho_0-\rho^*,u_0,\th_0-\th_\infty)\|_{H^{3}}\leq \delta.
	\eeq
	Then there exists a unique global solution $(\rho,u,\th)$ of initial value problem \eqref{ns3}--\eqref{initial-boundary} satisfying for any $t\geq0$,
\beq\label{energy-thm}
\begin{split}
\|(\rho-\rho^*,u,\th-\th_\infty)(t)\|_{H^3}^2
+\int_0^t\big(\|\nabla(\rho-\rho^*)\|_{H^{2}}^2+\|\nabla u\|_{H^3}^2+\|\nabla \th\|_{H^3}^2\big)ds
\leq C\|(\rho_0-\rho^*,u_0,\th_0-\th_\infty)\|_{H^3}^2,
\end{split}
\eeq
where $C$ is a positive constant independent of time $t$.
If further
\beno
\|(\rho_0-\rho^*,u_0,\th_0-\th_\infty)\|_{L^1}<\infty,
\eeno
then there exist constants $\delta_0>0$ and $\bar C_0>0$ such that for any $0<\delta\leq\delta_0$, we have
\beq\label{kdecay}
\|\na^k(\rho-\rho^*,u,\th-\th_\infty)(t)\|_{L^2}
\leq \bar C_0(1+t)^{-\f34-\f k2},\quad\text{for}~~k=0,1,2,3.
\eeq
\end{theo}

\begin{rema}
The global well-posedness theory of the full CNS equations with potential force in three-dimensional whole space was studied in \cite{duan-ukai-yang-zhao2007} under the $H^3-$ framework.
Furthermore, they also established the decay estimate \eqref{highdecayL1} if the initial data also belong to $L^1$.
Thus, the advantage of the decay rate \eqref{kdecay} in Theorem \ref{them3} is that the decay rate of the global solution $(\rho-\rho^*,u,\theta-\theta_\infty)$ itself and
its any order spatial derivative is optimal in the sense that it coincides with the rate of the solution of the linerized system.
\end{rema}

\begin{rema}
By using the Sobolev interpolation inequality, we can establish the following estimate:
\beno
\|\na^k(\rho-\rho^*,u,\th-\th_\infty)(t)\|_{L^p}
\leq \bar C_0(1+t)^{-\f32-\f k2+\f{3}{2p}},
\eeno
for all $2\leq p<+\infty$ and $k=0,1,2,3$. If $p=+\infty$, then $k=0,1$.
Therefore, the global solution $(\rho,u,\theta)$ of Cauchy problem \eqref{ns3}-\eqref{initial-boundary} tends to the constant equilibrium
state $(\rho^*,0,\th_\infty)$ in $L^\infty-$norm at the
$(1+t)^{-\f32}-$rate.
At the same time, we point out that the time derivative of density, velocity and temperature $(\rho_t,u_t,\th_t)$ tends to $(0,0,0)$ in $L^2-$norm at the $(1+t)^{-\f54}-$rate.
\end{rema}

\begin{rema}
    If the initial data $(\rho_0-\rho^*,u_0,\theta_0-\theta_\infty)\in H^N\cap L^1 (N\geq3)$, we also can get the similar decay result that
    \beq\label{kdecay2}
    \|\na^k(\rho-\rho^*,u,\th-\th_\infty)(t)\|_{L^2}
    \leq \bar C_0(1+t)^{-\f34-\f k2},\quad\text{for}~~k=0,\cdots,
    N.
    \eeq
    These decay rates of the solution itself and its spatial derivatives are optimal in the sense that it coincides with the rate of the solution of the heat equation.
    The decay estimate \eqref{kdecay2} was proven in this paper for $N=3$ (see Theorem \ref{them3}), however, the case $N>3$ can be handled in the same way and so we omit the proof.
\end{rema}

\textbf{To end this section, we would like to introduce our strategies for deriving the optimal time-decay rates for the full CNS equations with poteneial force.}
We can only obtain the lower dissipation estimate about the density, which is essentially caused by the degenerate dissipative structure of the system \eqref{ns3}-\eqref{initial-boundary} satisfying hyperbolic-parabolic coupling system.
Therefore, we will focus on how to obtain the energy estimates which include only the highest-order spatial derivative of the solution.
We point out that the equilibrium state of global solution will depend on the spatial variable caused by potential force.
This will also create some fundamental and additional difficulties in the process of the energy estimates, see Lemmas \ref{enn-1}, \ref{enn} and \ref{ennjc}.
We can derive in a similar way as \eqref{highdecayL1} in \cite{duan-ukai-yang-zhao2007}, by combining the energy estimate and the decay rate of linearized system, one can obtain the following decay estimates:
\begin{equation}\label{basic-decay-d}
\|\nabla^k (\rho-\rho^*,u,\theta-\theta_\infty)(t)\|_{H^{3-k}}
\lesssim (1+t)^{-(\frac34+\frac{k}{2})},\quad k=0,1,
\end{equation}
if the initial data $(\rho_0-\rho^*, u_0,\theta_0-\theta_\infty)$ belongs to $H^3 \cap L^1$.
We then prove that the decay estimate \eqref{basic-decay-d} for $k=2$ by using the basic decay estimate \eqref{basic-decay-d}.
Motivated by \cite{gao2021}, one can apply the time weighted method and the basic decay estimate \eqref{basic-decay-d} to achieve this goal.
Indeed, by virtue of the classical energy estimate, we can establish following estimate:
\beq\label{E32}
\begin{split}
	&(1+t)^{\f72+\ep_0}\mathcal{E}^{3}_{2}(t)
	+\int_0^t(1+\tau)^{\f72+\ep_0}\big(\|\na^{3}(\rho-\rho^*)\|_{L^{2}}^2
	+\|\na^{3}(u,\theta-\theta_\infty)\|_{H^{1}}^2\big)d\tau\\
	\lesssim&\|\na^{2}(\rho_0-\rho^*,u_0,\theta_0-\theta_\infty)\|_{H^{1}}^2
	+\int_0^t(1+\tau)^{\f52+\ep_0}\|\na^{2}(\rho-\rho^*, u,\theta-\theta_\infty)\|_{H^{1}}^2d\tau.
\end{split}
\eeq
where $\mathcal{E}^{3}_{2}(t)$ is equivalent to $\|\na^2(\rho-\rho^*, u,\theta-\theta_\infty)\|_{H^1}$. Thus, we need to control the second term on the right handside of \eqref{E32}.
With the help of the decay estimate for $k=1$ in \eqref{basic-decay-d}, we can obtain that\beq\label{E31}
\begin{split}
	(1+t)^{\f52+\ep_0} \mathcal{E}^3_1(t)	
	+\int_0^t(1+\tau)^{\f52+\ep_0}\big(\|\na^{2}(\rho-\rho^*)\|_{H^{1}}^2
	+\|\na^{2}(u,\theta-\theta_\infty)\|_{H^{2}}^2\big)d\tau\lesssim (1+t)^{\ep_0},
\end{split}
\eeq
where $\mathcal{E}^{3}_{1}(t)$ is equivalent to $\|\na(\rho-\rho^*, u,\theta-\theta_\infty)\|_{H^2}$. 
The combination of \eqref{E32} and decay estimate \eqref{E31} gives the decay estimate \eqref{basic-decay-d} for $k=2$ directly.

Due to the presence of potential force term $\rho \nabla \phi$, we can not apply the time weighted method mentioned above to build the optimal decay rate for the third order spatial derivative of global solution directly.
In order to overcome this difficult, motivated by \cite{wu2020}, we establish some energy estimate for the quantity
$\int_{|\xi|\geq\eta}\widehat{\na^{2}u}\cdot \overline{\widehat{\na^{3}(\rho-\rho^*)}}d\xi$,
namely the higher frequency part, rather than $\int \nabla^{2} u\cdot\nabla^{3}(\rho-\rho^*) dx$.
Here $\widehat{\na^{2}u}$ and $\widehat{\na^{3}(\rho-\rho^*)}$ stand for the Fourier part of
$\na^{2}u$ and $\na^{3}(\rho-\rho^*)$, respectively.
We point out that the advantage is that the quantity $\|\na^{3}(\rho-\rho^*, u,\theta-\theta_\infty)\|_{L^2}^2-\eta_2\int_{|\xi|\geq\eta}\widehat{\na^{2}u}\cdot \overline{\widehat{\na^{3}(\rho-\rho^*)}}d\xi$
is equivalent to $\|\na^{3}(\rho-\rho^*, u,\theta-\theta_\infty)\|_{L^2}^2$ for some small constant $\eta_2$.
Hence, by virtue of some energy estimate and decay estimate, one can obtain the following inequality:
\beq\label{highesthigh}
\begin{split}	&\f{d}{dt}\Big\{\|\na^{3}(\rho-\rho^*, u,\theta-\theta_\infty)\|_{L^2}^2-\eta_2\int_{|\xi|\geq\eta}\widehat{\na^{2}u}\cdot \overline{\widehat{\na^{3}(\rho-\rho^*)}}d\xi \Big\}\\
	&\quad+\|\na^{3}( u^h,(\theta-\theta_\infty)^h)\|_{L^2}^2+\eta_2\|\na^{3}(\rho-\rho^*)^h\|_{L^2}^2\\
	\lesssim&  \|\na^{3}((\rho-\rho^*)^l, u^l,(\theta-\theta_\infty)^l)\|_{L^2}^2+(1+t)^{-6}.
\end{split}
\eeq
Thus, one has to estimate the decay rate of the low-frequency term
$\|\na^{3}((\rho-\rho^*)^l, u^l,(\theta-\theta_\infty)^l)\|_{L^2}^2$.
Indeed, the combination of Duhamel's principle and decay estimate of $k-th$ $(0\leq k\leq 3)$ order spatial derivative of solution obtained above help us to build that\beno
\|\na^3((\rho-\rho^*)^l, u^l,(\theta-\theta_\infty)^l)\|_{L^2}\lesssim\delta\sup_{0\leq\tau\leq t}\|\na^3 (\rho-\rho^*, u,\theta-\theta_\infty)\|_{L^2}+(1+t)^{-\f94}.
\eeno
This, together with \eqref{highesthigh}, by using the smallness of $\d$, we can obtain the optimal decay rate for the third order spatial derivative of global solution to the full CNS equations with potential force.

\textbf{Notation:} Throughout this paper, for $1\leq p\leq +\infty$ and $s\in\mathbb{R}$, we simply denote $L^p(\mathbb{R}^3)$ and $H^s(\mathbb{R}^3)$ by $L^p$  and $H^s$, respectively.
And the constant $C$ denotes a general constant which may vary in different estimates.
$A\lesssim(\gtrsim) B$ stands for $A\leq(\geq) CB$ for some constant
$C>0$ independent of $A$ and $B$.
$A\sim B$ stands for $B\lesssim A\lesssim B$.
$\widehat{f}(\xi)=\mathcal F(f(x))$ represents the usual Fourier transform of the function $f(x)$ with respect to $x\in\mathbb{R}^3$. $\mathcal F^{-1}(\widehat{f}(\xi))$ means the inverse Fourier transform of $\widehat{f}(\xi)$ with respect to $\xi\in\mathbb{R}^3$. For the sake of simplicity, we write $\int f dx:=\int _{\mathbb{R}^3} f dx$.

The rest of the paper is organized as follows. In Section \ref{pre}, we introduce some important lemmas and basic fact, which will be useful in later analysis. Finally, the proof of Theorem \ref{them3} is given in Section \ref{rhox}.

\section{Preliminary}\label{pre}
In this section, we recall some elementary inequalities, which will be of frequency use in next section.
First of all, in order to deal with the terms about $\bar\rho(x)$ in the CNS equations with a potential force in energy estimate, we need the following Hardy inequality.
\begin{lemm}[Hardy inequality]\label{hardy}
	For $k\geq1$, suppose that $\f{\na\phi}{(1+|x|)^{k-1}}\in L^2$, then $\f{\phi}{(1+|x|)^{k}}\in L^2$, with the estimate
	\beno
	\begin{split}
		\|\f{\phi}{(1+|x|)^{k}}\|_{L^2}\leq C\|\f{\na\phi}{(1+|x|)^{k-1}}\|_{L^2}.
	\end{split}
	\eeno
\end{lemm}
The proof of Lemma \ref{hardy} is simply and we omit it here. We then introduce the following Sobolev interpolation of Gagliardo-Nirenberg inequality, which will be used extensively in energy estimates. The proof and more details may refer to \cite{guo2012}.
\begin{lemm}[Sobolev interpolation inequality]\label{inter}
	Let $2\leq p\leq +\infty$ and $0\leq l,k\leq m$. If $p=+\infty$, we require furthermore that $l\leq k+1$ and $m\geq k+2$. Then if $\na^l\phi\in L^2$ and $\na^m \phi\in L^2$, we have $\na^k\phi\in L^p$. Moreover, there exists a positive constant $C$ depending only on $k,l,m,p$ such that
	\begin{equation}\label{Sobolev}
	\|\na^k\phi\|_{L^p}\leq C\|\na^l\phi\|_{L^2}^{\theta}\|\na^m\phi\|_{L^2}^{1-\theta},
	\end{equation}
	where $0\leq\theta\leq1$ satisfying
	\beno
	\f k3-\f1p=\Big(\f l3-\f12\Big)\theta+\Big(\f m3-\f12\Big)(1-\theta).
	\eeno
\end{lemm}

Then we introduce the following commutator estimate, which will be useful in following energy estimates and can be found in \cite{majda2002} for more details.
\begin{lemm}\label{commutator}
	Let $k\geq1$ be an integer and define the commutator\beno
       [\na^k,f]g \overset{def}{=} \na^k(fg)-f\na^kg.
    \eeno
    Then we have
    \beno
    \|[\na^k,f]g\|_{L^2}\leq C\|\na f\|_{L^\infty}\|\na^{k-1}g\|_{L^2}+C\|\na^k f\|_{L^2}\|g\|_{L^\infty},
    \eeno
    where $C$ is a positive constant dependent only on $k$.
\end{lemm}

Finally, we end up this section with the following lemma. The proof and more details may refer to \cite{chen2021}.
\begin{lemm}\label{tt2}
Let $r_1,r_2>0$ be two real numbers, for any $0<\ep_0<1$, we have
\beno
\begin{split}
	\int_0^{\f t2}(1+t-\tau)^{-r_1}(1+\tau)^{-r_2}d\tau\leq C& \left\{\begin{array}{l}
		(1+t)^{-r_1},\quad \text{for}~~ r_2>1,\\
		(1+t)^{-r_1+\ep_0},\quad~~ \text{for}~~ r_2=1, \\
		(1+t)^{-(r_1+r_2-1)},\quad \text{for}~~ r_2<1,
	\end{array}\right.\\
	\int_{\f t2}^{t}(1+t-\tau)^{-r_1}(1+\tau)^{-r_2}d\tau\leq C& \left\{\begin{array}{l}
		(1+t)^{-r_2},\quad \text{for}~~ r_1>1,\\
		(1+t)^{-r_2+\ep_0},\quad~~ \text{for}~~ r_1=1, \\
		(1+t)^{-(r_1+r_2-1)},\quad \text{for}~~ r_1<1,
	\end{array}\right.
\end{split}
\eeno
where $C$ is a positive constant independent of time.
\end{lemm}

\section{Optimal decay of the full CNS equations with potential force}\label{rhox}

In this section, we will give the proof for Theorem \ref{them3} that include the global well-posedness theory and time decay estimates.
First of all, we noted that the global small solution of the full CNS equations can be proven directly by taking the strategy of standard energy method in \cite{{duan-ukai-yang-zhao2007},{mat1983}}, when the initial data is small perturbation near the equilibrium state.
Thus, one can assume that the global solution $(\rho, u,\theta)$ in Theorem \ref{them3} exists and satisfies the energy estimate \eqref{energy-thm},i.e.,
\beq\label{energy-thm-01}
\begin{split}
\|(\rho-\rho^*,u,\theta-\theta_{\infty})\|_{H^3}^2
+\int_0^t\big(\|\nabla(\rho-\rho^*)\|_{H^{2}}^2+\|\nabla u\|_{H^3}^2+\|\nabla \theta(t)\|_{H^{3}}^2\big)ds
\leq C\|(\rho_0-\rho^*,u_0,\theta_0-\theta_\infty)\|_{H^3}^2,
\end{split}
\eeq
for all $t\geq 0$.
Secondly, similar to the decay estimate \eqref{highdecayL1}, the combination of the energy estimate and the decay rate of linearized system can help us to establish the following decay estimates:
\begin{equation}\label{basic-decay}
\|\nabla^k (\rho-\rho^*)(t)\|_{H^{3-k}}+\|\nabla^k u(t)\|_{H^{3-k}}+\|\nabla^k (\theta-\theta_\infty)(t)\|_{H^{3-k}}
\le C(1+t)^{-(\frac34+\frac{k}{2})},\quad k=0,1,
\end{equation}
if the initial data $(\rho_0-\rho^*, u_0,\theta_0-\theta_\infty)$ belongs to $L^1$ additionally.
Now, we only focus on establishing the optimal decay rate for the higher order spatial derivative of solution. 
More precisely, we would like to prove that the decay estimate
\eqref{basic-decay} for the case $k=2,3$.
Thus, we define
$$n(t,x)\overset{def}{=}\rho(t,x)-\rho^*(x),\quad
\bar\rho(x)\overset{def}{=} \rho^*(x)-\rho_{\infty}, \quad
v(t,x)\overset{def}{=}\sqrt{\f{\rho_{\infty}}{p_1}}u(t,x),\quad q(t,x)\overset{def}{=}\sqrt{\f{p_2\rho_{\infty}}{p_1p_3}}(\theta(t,x)-\theta_\infty),
$$
where \[p_1=\f{p_{\rho}(\rho_\infty,\theta_\infty)}{\rho_\infty},\quad p_2=\f{p_{\theta}(\rho_\infty,\theta_\infty)}{\rho_\infty},\quad p_3=\f{\theta_\infty p_{\theta}(\rho_\infty,\theta_\infty)}{c_\nu\rho_\infty},  \]
then \eqref{ns3}--\eqref{initial-boundary} can be rewritten in the following perturbation form
\beq\label{ns5}
\left\{\begin{array}{lr}
	n_t +\gamma\dive v=F_1,\quad (t,x)\in \mathbb{R}_{+}\times \mathbb{R}^3,\\
	v_t+\gamma\na n+\bar\lam \na q-\mu_1\tri v-\mu_2\na\dive v =F_2,\quad (t,x)\in \mathbb{R}_{+}\times \mathbb{R}^3,\\
	q_t-\bar{\kappa}\tri q+\bar\lam \dive v=F_3,\quad (t,x)\in \mathbb{R}_{+}\times \mathbb{R}^3,\\
	(n,v,q)|_{t=0}\overset{def}{=}(n_0,v_0,q_0)=(\rho_0-\rho^*,\sqrt{\f{\rho_{\infty}}{P_1}}u_0,\sqrt{\f{P_2\rho_{\infty}}{P_1P_3}}(\theta_0-\theta_\infty))\rightarrow (0,0,0),~~\text{as}~~|x|\rightarrow \infty,
\end{array}\right.
\eeq
where
$\mu_1=\f{\mu}{\rho_{\infty}}, \mu_2=\f{\mu+\lam}{\rho_{\infty}},
\gamma=\sqrt{p_1\rho_{\infty}}$, $\bar\lam=\sqrt{p_2p_3}$, $\bar\kappa=\f{\kappa}{c_{\nu}\rho_\infty}$, and
\beno
\begin{split}
	F_1&=-\f{\mu_1\gamma}{\mu}\dive [(n+\bar\rho)v],\\
	F_2&=-\f{\mu_1\gamma}{\mu}v\cdot \na v-\f{\mu}{\mu_1}f(n+\bar\rho)\big(\mu_1\tri v+\mu_2\na\dive v\big)-g(n+\bar\rho,q) \na n -h(n,\bar\rho,q) \na\bar\rho-r(n+\bar\rho,q)\na q,\\
	F_3&=-\f{\mu_1\gamma}{\mu}v\cdot \na q+\bar\kappa f(n+\bar\rho)\tri q-m(n+\bar\rho,q)\dive v+\f{\bar\kappa\mu_1^2\gamma^2}{\kappa\mu^2}\sqrt{\f{p_2}{p_3}}\Psi(v)+\f{\bar\kappa\mu_1\gamma^2}{\kappa\mu}\sqrt{\f{p_2}{p_3}}f(n+\bar\rho)\Psi(v).
\end{split}
\eeno
Here the nonlinear functions $f, g, h$ and $m$ are defined by
\beno
\begin{split}
&f(n+\bar\rho)\overset{def}{=}\f{1}{n+\bar\rho+\rho_\infty}-\f{1}{\rho_\infty},\quad g(n+\bar\rho,q)\overset{def}{=}\f{\mu}{\mu_1\gamma}\Big(\f{p_{\rho}(n+\bar\rho+\rho_\infty,\sqrt{\f{p_2\rho_{\infty}}{p_1p_3}}q+\theta_{\infty})}{n+\bar\rho+\rho_\infty}-\f{p_{\rho}(\rho_{\infty},\theta_{\infty})}{\rho_{\infty}}\Big),\\
&h(n,\bar\rho,q)\overset{def}{=}\f{\mu}{\mu_1\gamma}\Big(\f{p_{\rho}(n+\bar\rho+\rho_\infty,\sqrt{\f{p_2\rho_{\infty}}{p_1p_3}}q+\theta_{\infty})}{n+\bar\rho+\rho_\infty}-\f{p_{\rho}(\bar\rho+\rho_{\infty},\theta_{\infty})}{\bar\rho+\rho_{\infty}}\Big),\\
&r(n+\bar\rho,q)\overset{def}{=}\f{\mu}{\mu_1\gamma}\Big(\f{p_{\theta}(n+\bar\rho+\rho_\infty,\sqrt{\f{p_2\rho_{\infty}}{p_1p_3}}q+\theta_{\infty})}{n+\bar\rho+\rho_\infty}-\f{p_{\theta}(\rho_{\infty},\theta_{\infty})}{\rho_{\infty}}\Big),\\
&m(n+\bar\rho,q)\overset{def}{=}\f{1}{c_\nu}\sqrt{\f{p_2}{p_3}}\Big(\f{(\sqrt{\f{p_2\rho_{\infty}}{p_1p_3}}q+\theta_\infty)p_{\theta}(n+\bar\rho+\rho_\infty,\sqrt{\f{p_2\rho_{\infty}}{p_1p_3}}q+\theta_{\infty})}{n+\bar\rho+\rho_\infty}-\f{\theta_\infty p_{\theta}(\rho_{\infty},\theta_{\infty})}{\rho_{\infty}}\Big).
\end{split}
\eeno

\subsection{Energy estimates}
In this subsection, we would like to establish the following differential inequality, which is the key to obtain the optimal decay rate
for the higher order spatial derivative of solution.
First of all, let us define the energy $\mathcal{E}^3_l(t)$ as
$$
\mathcal{E}^3_l(t)\overset{def}{=}
\sum_{k=l}^{3}\|\na^k(n,v,q)\|_{L^2}^2
+\eta_1\sum_{k=l}^{2}\int\na^{k}v\cdot\na^{k+1}n dx, \quad 0\le l \le 3,
$$
where $\eta_1$ is a small positive constant.
The smallness of parameter $\eta_1$ lead us to obtain the following equivalent realtion
\begin{equation}\label{emleq}
c_1 \|\nabla^l(n, v,q)\|_{H^{3-l}}^2
\le \mathcal{E}^3_l(t)
\le c_2 \|\nabla^l(n, v,q)\|_{H^{3-l}}^2 .
\end{equation}
where $c_1$ and $c_2$ are positive constants independent of time.
And one can deduce from the relation \eqref{p} and the condition \eqref{phik} given in Theorem \ref{them3} that
\beq\label{condition1}
\begin{split}
\sum_{k=0}^{4}\|(1+|x|)^{k}\na^k(\rho^*-\rho_{\infty})\|_{L^2\cap L^\infty}\leq \delta.
\end{split}
\eeq
The combination of \eqref{condition1} and the Sobolev interpolation inequality help us to deduce that
\begin{equation}\label{density-control}
\sum_{k=0}^{4}\|(1+|x|)^{k}\na^k(\rho^*-\rho_{\infty})\|_{L^p}\leq \delta,
\quad 2\leq p\leq+\infty.
\end{equation}
This inequality will be used frequently in energy estimate in this section.
Now we state the main result in this subsection.

\begin{prop}\label{nenp}
Under the assumptions in Theorem \ref{them3}, for any $l=1,2$, it holds that
\beq\label{eml}
\begin{split}	\f{d}{dt}\mathcal{E}^3_l(t)+\eta_1\|\na^{l+1}n\|_{H^{2-l}}^2
+\|\na^{l+1}(v,q)\|_{H^{3-l}}^2\leq 0,
\end{split}
\eeq
where $\eta_1$ is a small positive constant.
\end{prop}
Recalling the energy estimate \eqref{energy-thm-01},
then there exists a positive constant $C$ such that for any $k\geq 1$,
\beno
\begin{split}
	&|f(n+\bar\rho)|\leq C|n+\bar\rho|,\quad |g(n+\bar\rho,q)|\leq C|n+\bar\rho+q|,\quad |h(n,\bar\rho,q)|\leq C|n+q|,\\
	&|r(n+\bar\rho,q)|\leq C|n+\bar\rho+q|,\quad |m(n+\bar\rho,q)|\leq C|n+\bar\rho+q|,\\
	&
	|f^{(k)}(n+\bar\rho)|\leq C,\quad |g^{(k)}(n+\bar\rho,q)|\leq C,\quad|h^{(k)}(n,\bar\rho,q)|\leq C,\quad|r^{(k)}(n+\bar\rho,q)|\leq C,\quad |m^{(k)}(n+\bar\rho,q)|\leq C.
\end{split}
\eeno
Thus, it is easy to check that\beq\label{F-defi}
\begin{split}
	F_1\sim& \dive ((n+\bar\rho)v),\\
	F_2\sim& v\cdot \na v+(n+\bar\rho)(\tri v+\na\dive v)+(n+q)\na(n+q)+\na(\bar\rho(n+q)),\\
	F_3\sim&v\cdot\na q+(n+\bar\rho)\tri q+(n+\bar\rho)\dive v+q\dive v+(n+\bar\rho)\Psi(v)+\Psi(v).
\end{split}
\eeq

Next, we will give the following three lemmas, which will play vital role to prove Proposition \ref{nenp}. The first one is Lemma \ref{enn-1} concerning the basic energy estimate for $k-th$ $(k=1,2)$ order spatial derivative of solution.

\begin{lemm}\label{enn-1}
Under the assumptions in Theorem \ref{them3}, for $k=1,2$, it holds that
\beq\label{en1}
\f{d}{dt}\|\na^k(n,v,q)\|_{L^2}^2+\|\na^{k+1}(v,q)\|_{L^2}^2
\leq C\delta \|\na^{k+1}(n,v,q)\|_{L^2}^2,
\eeq
where $C$ is a positive constant independent of time.
\end{lemm}
\begin{proof}
We apply differential operator $\na^k$ to $\eqref{ns5}_1$, $\eqref{ns5}_2$ and $\eqref{ns5}_3$, multiply the resulting equations by $\na^kn$, $\na^kv$ and $\na^kq$, respectively, then integrate over $\mathbb{R}^3$, to find\beq\label{ehk}
\begin{split}
	&\f12\f{d}{dt}\|\na^k(n,v,q)\|_{L^2}^2+\mu_1\|\na^{k+1} v\|_{L^2}^2+\mu_2\|\na^k\dive v\|_{L^2}^2+\bar\kappa\|\na^{k+1} q\|_{L^2}^2\\
	=& \int \na^kF_1\cdot \na^kn dx+\int \na^kF_2\cdot \na^kv dx+\int \na^kF_3\cdot \na^kq dx.
\end{split}\eeq
For $k=1$, integrating by part and using the definition of $F_1$ and the estimates \eqref{density-control} and $\eqref{F-defi}_1$, one can obtain from Sobolev and Hardy inequalities that \beq\label{F11}
\begin{split}
\int \na F_1\cdot \na n dx \leq& C\|F_1\|_{L^2}\|\na^2n\|_{L^2}\\
\leq &C\int ( v\cdot\na n+n\dive v  )\cdot \na^{2}n dx
     +C\int (v\cdot\na\bar\rho+\bar\rho \dive v)\cdot \na^{2}n dx\\
\leq& C\big(\|v\|_{L^3}\|\na n\|_{L^6}+\|n\|_{L^3}\|\na v\|_{L^6}+\|(1+|x|)\na\bar\rho\|_{L^3}\|\f{v}{1+|x|}\|_{L^6}+\|\bar\rho\|_{L^3}\|\na v\|_{L^6}\big)\|\na^2n\|_{L^2}\\
\leq& C(\|(n,v)\|_{H^1}+\delta)\|\na^2(n,v)\|_{L^2}^2\leq C\delta\|\na^2(n,v)\|_{L^2}^2.
\end{split}
\eeq
It then follows from integrating by parts, Sobolev and Hardy inequalities, the estimates \eqref{density-control} and $\eqref{F-defi}_2$ that\beq\label{F21}
\begin{split}
    \int \na F_2\cdot \na v dx
    \leq& C\|F_2\|_{L^2}\|\na^2v\|_{L^2}\\
    \leq& C\int \big(v\cdot\na v+n\tri v+n\na\dive v+n\na n+q\na n+n\na q+q\na q\big)\na^2vdx\\
    &+C\int \big(\bar\rho\tri v+\bar\rho\na\dive v+\bar\rho\na n+n\na \bar\rho+q\na \bar\rho+\bar\rho\na q\big)\na^2vdx\\
	\leq& C\big(\|v\|_{L^3}\|\na v\|_{L^6}+\|n\|_{L^\infty}\|\na^2 v\|_{L^2}+\|n\|_{L^3}\|\na n\|_{L^6}+\|q\|_{L^3}\|\na n\|_{L^6}+\|n\|_{L^3}\|\na q\|_{L^6}\\
	&+\|q\|_{L^3}\|\na q\|_{L^6}+\|\bar\rho\|_{L^\infty}\|\na^2 v\|_{L^2}+\|\bar\rho\|_{L^3}\|\na (n,q)\|_{L^6}\\
&+\|(1+|x|)\na\bar\rho\|_{L^3}\|\f{(n,q)}{1+|x|}\|_{L^6}\big)\|\na^2v\|_{L^2}\\
	\leq& C(\|(n,v,q)\|_{H^2}+\delta)\|\na^2(n,v,q)\|_{L^2}^2\leq C\delta\|\na^2(n,v,q)\|_{L^2}^2.
\end{split}
\eeq
The application of integration by parts, Sobolev and Hardy inequalities, the estimates \eqref{density-control} and $\eqref{F-defi}_3$ implies directly
\beq\label{F31}
\begin{split}
	&\int \na F_3\cdot \na q dx\leq C\|F_3\|_{L^2}\|\na^2q\|_{L^2}\\
	\leq& C\int \big(v\cdot\na q+n\tri q+n\dive v+q\dive v+n\Psi(v)+\Psi(v)\big)\na^2qdx+C\int \big(\bar\rho\tri q+\bar\rho\dive v+\bar\rho\Psi(v)\big)\na^2qdx\\
	\leq&C\big(\|v\|_{L^3}\|\na q\|_{L^6}+\|n\|_{L^\infty}\|\na^2 q\|_{L^2}+\|n\|_{L^3}\|\na v\|_{L^6}+\|q\|_{L^3}\|\na v\|_{L^6}+\|n\|_{L^\infty}\|\na v\|_{L^3}\|\na v\|_{L^6}\\
	&+\|\na v\|_{L^3}\|\na v\|_{L^6}+\|\bar\rho\|_{L^\infty}\|\na^2 q\|_{L^2}+\|\bar\rho\|_{L^3}\|\na v\|_{L^6}+\|\bar\rho\|_{L^\infty}\|\na v\|_{L^3}\|\na v\|_{L^6}\big)\|\na^2q\|_{L^2}\\
	\leq& C(\|(n,v,q)\|_{H^2}+\delta)\|\na^2(n,v,q)\|_{L^2}^2\leq C\delta\|\na^2(n,v,q)\|_{L^2}^2.
\end{split}
\eeq
Substituting three estimates \eqref{F11}-\eqref{F31} into \eqref{ehk} for $k=1$, it holds true\beq\label{energy-na1}
\f{d}{dt}\|\na(n,v,q)\|_{L^2}^2+\|\na^{2} (v,q)\|_{L^2}^2\leq C\delta\|\na^2(n,v,q)\|_{L^2}^2.
\eeq
As for $k=2$, the Sobolev and Hardy inequalities yields directly
\beq\label{F12}
\begin{split}
	&\int \na^2 F_1\cdot \na^2 n dx=-\int \na F_1\cdot \na^3 n dx\leq C\|\na F_1\|_{L^2}\|\na^3n\|_{L^2}\\
	\leq& C\big(\|(n,v)\|_{L^3}\|\na^2 (n,v)\|_{L^6}+\|\na(n,v)\|_{L^3}\|\na (n,v)\|_{L^6}+\|\bar\rho\|_{L^3}\|\na^2 v\|_{L^6}\\
	&\quad+\|(1+|x|)\na\bar\rho\|_{L^3}\|\f{\na v}{1+|x|}\|_{L^6}+\|(1+|x|)^2\na^2\bar\rho\|_{L^3}\|\f {v}{(1+|x|)^2}\|_{L^6}\big)\|\na^3n\|_{L^2}\\
	\leq& C(\|(n,v)\|_{H^2}+\delta)\|\na^3(n,v)\|_{L^2}^2\leq C\delta\|\na^3(n,v)\|_{L^2}^2,
\end{split}
\eeq
where we have used the estimate \eqref{density-control} and the fact that\beq\label{na3na6}
\|\na(n,v)\|_{L^3}\|\na (n,v)\|_{L^6}\leq C\|\na^{\f34}(n,v)\|_{L^2}^{\f23}\|\na^3 (n,v)\|_{L^2}^{\f13}\|(n,v)\|_{L^2}^{\f13}\|\na^3 (n,v)\|_{L^2}^{\f23}\leq C\delta \|\na^3 (n,v)\|_{L^2}.
\eeq
By virtue of integration by parts and Sobolev inequality, we find\beq\label{na2F2-}
\begin{split}
	\int \na^2 F_2\cdot \na^2 v dx=-\int \na F_2\cdot \na^3 v dx\leq C\|\na F_2\|_{L^2}\|\na^3v\|_{L^2}.
\end{split}\eeq
With the help of the definition of $F_2$, we have\beq\label{na2F2--}
\begin{split}
	\|\na F_2\|_{L^2}
	\leq& C\big(\|\na(v\cdot\na v)\|_{L^2}+\|\na(f(n+\bar\rho)\na^2 v)\|_{L^2}+\|\na(g(n+\bar\rho,q)\na n)\|_{L^2}+\|\na(h(n,\bar\rho,q)\na \bar\rho)\|_{L^2}\\
	&\quad+\|\na(r(n+\bar\rho,q)\na q)\|_{L^2}\big)\\
	\overset{def}{=}&I_1+I_2+I_3+I_4+I_5.
\end{split}\eeq
Sobolev inequality and \eqref{na3na6} yield that
\beno
I_1\leq C\|v\|_{L^3}\|\na^2 v\|_{L^6}+C\|\na v\|_{L^3}\|\na v\|_{L^6}\leq C\delta \|\na^3v\|_{L^2}.
\eeno
In view of the estimate \eqref{density-control}, Sobolev and Hardy inequalities, we deduce that\beno\begin{split}
I_2\leq& C\|f(n+\bar\rho)\|_{L^\infty}\|\na^3v\|_{L^2}+C\|\na f(n+\bar\rho)\|_{L^3}\|\na^2v\|_{L^6}\\
\leq& C\|(n+\bar\rho)\|_{L^\infty}\|\na^3v\|_{L^2}+C\|\na (n+\bar\rho)\|_{L^3}\|\na^3v\|_{L^2}\\
\leq& C\delta\|\na^3v\|_{L^2}.
\end{split}
\eeno
The application of the estimates \eqref{density-control} and \eqref{na3na6}, Sobolev and Hardy inequalities yields directly\beno\begin{split}
	I_3\leq& C\|g(n+\bar\rho,q)\|_{L^3}\|\na^2n\|_{L^6}+C\|\na g(n+\bar\rho,q)\na n\|_{L^2}\\
	\leq& C\|(n+\bar\rho+q)\|_{L^3}\|\na^2n\|_{L^6}+C\|\na (n+q)\|_{L^3}\|\na n\|_{L^6}+C\|(1+|x|)\na \bar\rho\|_{L^3}\|\f{\na n}{1+|x|}\|_{L^6}\\
	\leq& C\delta\|\na^3(n,q)\|_{L^2}.
\end{split}
\eeno
By the estimate \eqref{density-control}, Hardy and Sobolev inequalities, it is easy to check that\beno
\begin{split}
	I_4\leq& C\|\f{h(n,\bar\rho,q)}{(1+|x|)^2}\|_{L^6}\|(1+|x|)^2\na^2\bar\rho\|_{L^3}+C\|\f {\na h(n,\bar\rho,q)}{1+|x|}\|_{L^6}\|(1+|x|)\na\bar\rho\|_{L^3}\\
	\leq& C\|\f{n+q}{(1+|x|)^2}\|_{L^6}\|(1+|x|)^2\na^2\bar\rho\|_{L^3}+C\|\f {\na(n+q)}{1+|x|}\|_{L^6}\|(1+|x|)\na\bar\rho\|_{L^3}\\
	\leq& C\delta\|\na^3(n,q)\|_{L^2}.
\end{split}
\eeno
Similar to the estimate of $I_3$, we apply the estimates \eqref{density-control} and \eqref{na3na6}, Sobolev and Hardy inequalities to obtain\beno\begin{split}
	I_5
	\leq& C\|(n+\bar\rho+q)\|_{L^3}\|\na^2q\|_{L^6}+C\|\na (n+q)\|_{L^3}\|\na q\|_{L^6}+C\|(1+|x|)\na \bar\rho\|_{L^3}\|\f{\na q}{1+|x|}\|_{L^6}\\
	\leq& C\delta\|\na^3(n,q)\|_{L^2}.
\end{split}
\eeno
Inserting the estimates of terms $I_1$ to $I_5$ into \eqref{na2F2--}, it follows immediately\beno
\begin{split}
	\|\na F_2\|_{L^2}\leq& C\delta\|\na^3(n,v,q)\|_{L^2}.
\end{split}\eeno
Substituting this estimate into \eqref{na2F2-}, we have
\beq\label{na2F2}
\begin{split}
	\int \na^2 F_2\cdot \na^2 v dx\leq C\delta\|\na^3(n,v,q)\|_{L^2}^2.
\end{split}\eeq
Integration by parts and Sobolev inequality imply that\beq\label{na2F3-}
\begin{split}
	\int \na^2 F_3\cdot \na^2 q dx=-\int \na F_3\cdot \na^3 q dx\leq C\|\na F_3\|_{L^2}\|\na^3q\|_{L^2}.
\end{split}\eeq
Remebering the definition of $F_2$, we find\beq\label{na2F3--}
\begin{split}
	\|\na F_3\|_{L^2}
	\leq& C\big(\|\na(v\cdot\na q)\|_{L^2}+\|\na(f(n+\bar\rho)\na^2 q)\|_{L^2}+\|\na(f(n+\bar\rho,q)\Psi(v))\|_{L^2}+\|\na\Psi(v)\|_{L^2}\\
	&\quad+\|\na(m(n+\bar\rho,q)\dive v)\|_{L^2}\big)\\
	\overset{def}{=}&J_1+J_2+J_3+J_4+J_5.
\end{split}\eeq
According to the Sobolev inequality and the estimate \eqref{na3na6}, we deduce that\beno
J_1\leq C\|v\|_{L^3}\|\na^2 q\|_{L^6}+C\|\na v\|_{L^3}\|\na q\|_{L^6}\leq C\delta \|\na^3(v,q)\|_{L^2}.
\eeno
By Sobolev inequality and the estimate \eqref{density-control}, one can deduce directly that\beno\begin{split}
	J_2\leq& C\|f(n+\bar\rho)\|_{L^\infty}\|\na^3q\|_{L^2}+C\|\na f(n+\bar\rho)\|_{L^3}\|\na^2q\|_{L^6}\\
	\leq& C\|(n+\bar\rho)\|_{L^\infty}\|\na^3q\|_{L^2}+C\|\na (n+\bar\rho)\|_{L^3}\|\na^3q\|_{L^2}\\
	\leq& C\delta\|\na^3q\|_{L^2}.
\end{split}
\eeno
Using the Sobolev and Hardy inequalities, the estimate \eqref{density-control} and \eqref{na3na6}, it holds\beno
\begin{split}
	J_3\leq& C\|f(n+\bar\rho)\|_{L^\infty}\|\na v\|_{L^3}\|\na^2 v\|_{L^6}+C\|\na f(n+\bar\rho)\|_{L^\infty}\|\na v\|_{L^3}\|\na v\|_{L^6}\\
	\leq& C\|(n+\bar\rho)\|_{L^\infty}\|\na v\|_{L^3}\|\na^3 v\|_{L^2}+C\|\na(n+\bar\rho)\|_{L^\infty}\|\na v\|_{L^3}\|\na v\|_{L^6}\\
	\leq& C\delta\|\na^3v\|_{L^2}.
\end{split}
\eeno
According to Sobolev inequality and the estimate \eqref{na3na6}, we obtain immediately\beno
\begin{split}
	J_4\leq C\|\na v\|_{L^3}\|\na^2 v\|_{L^6}\leq C\delta\|\na^3 v\|_{L^2}.
\end{split}
\eeno
In view of the estimates \eqref{density-control} and \eqref{na3na6}, Sobolev and Hardy inequalities, one deduces that\beno
\begin{split}
	J_5\leq& C\|m(n+\bar\rho,q)\|_{L^3}\|\na^2v\|_{L^6}+C\|\na m(n+\bar\rho,q)\na v\|_{L^2}\\
	\leq& C\|(n+\bar\rho+q)\|_{L^3}\|\na^3v\|_{L^2}+C\|\na(n+q)\|_{L^3}\|\na v\|_{L^6}+C\|(1+|x|)\na\bar\rho\|_{L^3}\|\f{\na v}{1+|x|}\|_{L^6}\\
	\leq& C\delta\|\na^3(n,v,q)\|_{L^2}.
\end{split}
\eeno
Then, the combination of the estimates of terms $J_1$ to $J_5$ and \eqref{na2F3--} implies directly\beno
\|\na F_3\|_{L^2}\leq C\delta\|\na^3(n,v,q)\|_{L^2}.
\eeno
We substitute this estimate into \eqref{na2F3-}, to find
\beq\label{na2F3}\begin{split}
	\int \na^2 F_3\cdot \na^2 q dx
    \leq C\delta\|\na^3(n,v,q)\|_{L^2}^2.
\end{split}
\eeq
Substituting \eqref{F12}, \eqref{na2F2} and \eqref{na2F3} into \eqref{ehk} for $k=2$, we find
\beno
\f{d}{dt}\|\na^2(n,v,q)\|_{L^2}^2+\|\na^{3} (v,q)\|_{L^2}^2\leq C\delta\|\na^3(n,v,q)\|_{L^2}^2,
\eeno
which, together with \eqref{energy-na1}, gives \eqref{en1} directly.
	Thus, we complete the proof of this lemma.
\end{proof}

We then derive the energy estimate for third order spatial derivative of the solution.
\begin{lemm}\label{enn}
	Under the assumptions in Theorem \ref{them3}, it holds that
	\beq\label{en2}
	\f{d}{dt}\|\na^{3}(n,v,q)\|_{L^2}^2+\|\na^{4}(v,q)\|_{L^2}^2\leq C\delta \|(\na^{3}n,\na^{4}v,\na^4q)\|_{L^2}^2,
	\eeq
	where $C$ is a positive constant independent of time.
\end{lemm}
\begin{proof}
Applying differential operator $\na^{3}$ to $\eqref{ns5}_1$, $\eqref{ns5}_2$ and $\eqref{ns5}_3$,
multiplying the resulting equations by $\na^{3}n$, $\na^{3}v$ and $\na^{3}q$, respectively,
and integrating over $\mathbb{R}^3$, it holds
\beq\label{ehk1}
\begin{split}
	&\f12\f{d}{dt}\|\na^{3}(n,v,q)\|_{L^2}^2+\mu_1\|\na^{4} v\|_{L^2}^2+\mu_2\|\na^{3}\dive v\|_{L^2}^2+\bar\kappa\|\na^{4} q\|_{L^2}^2\\
	=& \int \na^{3}F_1\cdot \na^{3}n dx+\int \na^{3}F_2\cdot \na^{3}v dx+\int \na^{3}F_3\cdot \na^{3}q dx.
\end{split}
\eeq
Now we estimate three terms on the right handside of \eqref{ehk1} separately. In view of the definition of $F_1$, Sobolev inequality Lemma \ref{commutator} and integration by parts, we have
\beq\label{F13-}
\begin{split}
\int \na^{3}F_1\cdot \na^{3}n dx
=
&C\int \na^{3}( v\cdot\na n+n\dive v+v\cdot\na\bar\rho+\bar\rho \dive v)\cdot \na^{3}n dx\\
\leq&-C\int \dive v|\na^3n|^2dx+\big(\|[\na^3,v]\cdot \na n\|_{L^2}+\|\na^3(n\dive v)\|_{L^2}\\
&\quad+\|\na^3(v\cdot\na\bar\rho+\bar\rho\dive v)\|_{L^2}\big)\|\na^3 n\|_{L^2}.
\end{split}
\eeq
It is easy to check that\beq\label{F131}
\begin{split}
	\|[\na^3,v]\cdot \na n\|_{L^2}\leq C\|\na v\|_{L^\infty}\|\na^3n\|_{L^2}+C\|\na^2v\|_{L^3}\|\na^2 n\|_{L^6}+C\|\na^3v\|_{L^6}\|\na n\|_{L^3}\leq C\delta \|(\na^3n,\na^4v)\|_{L^2}.
\end{split}
\eeq
It then follows from Sobolev inequality that\beq\label{F132}
\begin{split}
	\|\na^3(n\dive v)\|_{L^2}\leq& C\|n\|_{L^\infty}\|\na^4v\|_{L^2}+C\|\na^3n\|_{L^2}\|\na v\|_{L^\infty}\leq C\delta\|(\na^3n,\na^4v)\|_{L^2}.
\end{split}\eeq
The application of Sobolev and Hardy inequalities yields directly
\beq\label{F133}\begin{split}
	&\|\na^3(v\cdot\na\bar\rho+\bar\rho\dive v)\|_{L^2}\\
	\leq &C\sum_{l=0}^{3}\Big(\|(1+|x|)^{4-l}\na^{4-l}\bar\rho\|_{L^\infty}\|\f{\na^lv}{(1+|x|)^{4-l}}\|_{L^2}+\|(1+|x|)^{l}\na^{l}\bar\rho\|_{L^\infty}\|\f{\na^{4-l}v}{(1+|x|)^{l}}\|_{L^2}\Big)\\
	\leq& C\delta\|\na^4v\|_{L^2}.
\end{split}
\eeq
Hence, the combination of the estimates \eqref{F13-}-\eqref{F133} gives
\beq\label{F13}
\begin{split}
	\int \na^{3}F_1\cdot \na^{3}n dx
	\leq C\delta \|(\na^{3}n,\na^{4}v)\|_{L^2}^2.
\end{split}
\eeq
Integration by parts, by use of Sobolev inequality gives\beq\label{F23-}
\int \na^{3}F_2\cdot \na^{3}v dx=-\int \na^{2}F_2\cdot \na^{4}v dx\leq C\|\na^2F_2\|_{L^2}\|\na^4v\|_{L^2}.
\eeq
In view of the definition of $F_2$, we have\beq\label{F23--}\begin{split}
	\|\na^2 F_2\|_{L^2}
	\leq& C\big(\|\na^2(v\cdot\na v)\|_{L^2}+\|\na^2(f(n+\bar\rho)\na^2 v)\|_{L^2}+\|\na^2(g(n+\bar\rho,q)\na n)\|_{L^2}\\
	&\quad+\|\na^2(h(n,\bar\rho,q)\na \bar\rho)\|_{L^2}+\|\na^2(r(n+\bar\rho,q)\na q)\|_{L^2}\big)\\
	\overset{def}{=}&K_1+K_2+K_3+K_4+K_5.
\end{split}
\eeq
Thanks to the commutator estimate in Lemma \ref{commutator} and Sobolev inequality, one can deduce that\beno
\begin{split}
	K_1\leq C\|v\|_{L^3}\|\na^3 v\|_{L^6}+C\|[\na^2,v]\cdot\na v\|_{L^2}\leq C\|v\|_{L^3}\|\na^4 v\|_{L^2}+C\|\na v\|_{L^3}\|\na^2v\|_{L^6}\leq C\delta\|\na^4 v\|_{L^2},
\end{split}
\eeno
where we have used the following estimate in the last inequality\beq\label{navna2v}
\begin{split}
	\|\na v\|_{L^3}\|\na^2v\|_{L^6}\leq C\|\na^{\f23}v\|_{L^2}^{\f34}\|\na^4v\|_{L^2}^{\f14}\|v\|_{L^2}^{\f14}\|\na^4v\|_{L^2}^{\f34}\leq C\delta\|\na^4v\|_{L^2}.
\end{split}
\eeq
According to the estimate \eqref{density-control}, Sobolev and Hardy inequalities, we deduce that\beno
\begin{split}
	K_2\leq& C\|(n+\bar\rho)\|_{L^\infty}\|\na^4 v\|_{L^2}+C\|\na^2 n\|_{L^6}\|\na^2v\|_{L^3}+C\|\na n\|_{L^\infty}\|\na n\|_{L^3}\|\na^2 v\|_{L^6}\\
	&\quad+C\|(1+|x|)^2\na^2\bar\rho\|_{L^\infty}\|\f{\na^2v}{(1+|x|)^2}\|_{L^2}+C\|\na n\|_{L^3}\|(1+|x|)\na\bar\rho\|_{L^\infty}\|\f{\na^2v}{1+|x|}\|_{L^6}\\
	&\quad+C\|(1+|x|)\na\bar\rho\|_{L^\infty}^2\|\f{\na^2v}{(1+|x|)^2}\|_{L^2}\\
	\leq& C\delta\|(\na^3 n,\na^4 v)\|_{L^2},
\end{split}
\eeno
where we have used the following estimtes in the last inequality\beq\label{na3na26}
\begin{split}
	\|\na n\|_{L^3}\|\na^2 v\|_{L^6}\leq \|\na n\|_{L^2}^{\f34}\|\na^3 n\|_{L^2}^{\f14}\|v\|_{L^2}^{\f14}\|\na^4 v\|_{L^2}^{\f34}\leq C\delta \|(\na^3 n,\na^4 v)\|_{L^2}.
\end{split}
\eeq
Using the estimate \eqref{density-control}, Sobolev and Hardy inequalities, it holds\beno
\begin{split}
	K_3\leq& C\|(n+\bar\rho+q)\|_{L^\infty}\|\na^3 n\|_{L^2}+C\|\na^2 n\|_{L^6}\|\na n\|_{L^3}+C\|\na^2 q\|_{L^6}\|\na n\|_{L^3}+\|\na (n,q)\|_{L^\infty}\|\na(n,q)\|_{L^3}\|\na n\|_{L^6}\\
	&\quad+C\|(1+|x|)^2\na^2\bar\rho\|_{L^\infty}\|\f{\na n}{(1+|x|)^2}\|_{L^2}+C\|\na (n,q)\|_{L^3}\|(1+|x|)\na\bar\rho\|_{L^\infty}\|\f{\na n}{1+|x|}\|_{L^6}\\
	&\quad+C\|(1+|x|)\na\bar\rho\|_{L^\infty}^2\|\f{\na n}{(1+|x|)^2}\|_{L^2}\\
	\leq& C\delta\|(\na^3 n,\na^4 v,\na^4q)\|_{L^2},
\end{split}
\eeno
where we have used the estimate \eqref{na3na26} and the following estimtes\beq\label{nannanq}
\begin{split}
	\|\na n\|_{L^3}\|\na n\|_{L^6}\leq& \|\na^{\f34} n\|_{L^2}^{\f23}\|\na^3 n\|_{L^2}^{\f13}\|n\|_{L^2}^{\f13}\|\na^3 n\|_{L^2}^{\f23}\leq C\delta \|\na^3 n\|_{L^2},\\
	\|\na q\|_{L^3}\|\na n\|_{L^6}\leq&\|\na^{\f14} q\|_{L^2}^{\f23}\|\na^4 q\|_{L^2}^{\f13}\|n\|_{L^2}^{\f13}\|\na^3 n\|_{L^2}^{\f23}\leq C\delta \|(\na^3 n,\na^4q)\|_{L^2}.
\end{split}
\eeq
We can use Sobolev and Hardy inequalities, the estimates \eqref{navna2v} and \eqref{nannanq} to find\beno
\begin{split}
	K_4\leq& C\|(1+|x|)^3\na^3\bar\rho\|_{L^\infty}\|\f{ n}{(1+|x|)^3}\|_{L^2}+C\|(1+|x|)^3\na^3\bar\rho\|_{L^3}\|\f{ q}{(1+|x|)^3}\|_{L^6}+C\|(1+|x|)\na\bar\rho\|_{L^\infty}\|\f{\na^2 n}{1+|x|}\|_{L^2}\\
	&\quad+C\|(1+|x|)\na\bar\rho\|_{L^3}\|\f{ \na^2q}{1+|x|}\|_{L^6}+\|\na\bar\rho\|_{L^\infty}\|\na (n,q)\|_{L^3}\|\na n\|_{L^6}+\|(1+|x|)\na\bar\rho\|_{L^\infty}\|\na q\|_{L^3}\|\f{\na q}{1+|x|}\|_{L^6}\\
	\leq& C\delta\|(\na^3 n,\na^4q)\|_{L^2}.
\end{split}
\eeno
It then follows from the estimate \eqref{nannanq}, Sobolev and Hardy inequalities that\beno
\begin{split}
	K_5\leq&C\|(n+\bar\rho+q)\|_{L^3}\|\na^3 q\|_{L^6}+C\|\na^2 (n,q)\|_{L^6}\|\na q\|_{L^3}+C\|\na (n,q)\|_{L^\infty}\|\na (n,q)\|_{L^6}\|\na q\|_{L^3}\\
	&\quad+C\|(1+|x|)^2\na^2\bar\rho\|_{L^3}\|\f{\na q}{(1+|x|)^2}\|_{L^6}+C\|\na\bar\rho\|_{L^\infty}\|\na (n,q)\|_{L^6}\|{\na q}\|_{L^3}\\
	&\quad+C\|(1+|x|)\na\bar\rho\|_{L^\infty}\|(1+|x|)\na\bar\rho\|_{L^3}\|\f{\na q}{(1+|x|)^2}\|_{L^6}\\
	\leq& C\delta\|(\na^3 n,\na^4q)\|_{L^2},
\end{split}
\eeno
where we have used the following estimtes\beq\label{navnav}
\begin{split}
	\|\na q\|_{L^3}\|\na q\|_{L^6}\leq\|q\|_{L^2}^{\f12}\|\na^4 q\|_{L^2}^{\f12}\|\na q\|_{L^2}^{\f12}\|\na^4 q\|_{L^2}^{\f12}\leq C\delta \|\na^4 q\|_{L^2}.
\end{split}
\eeq
Substituting the estimates of terms $K_1$ to $K_5$ into \eqref{F23--}, we find\beq\label{F2-estimate1}\begin{split}
	\|\na^2 F_2\|_{L^2}\leq C\delta \|(\na^3n,\na^4v,\na^4 q)\|_{L^2}.
\end{split}
\eeq
Hence, the combination of the estimate above and \eqref{F23-} implies directly
\beq\label{F23}
\begin{split}
	\int \na^{3}F_2\cdot \na^{3}v dx
	\leq C\delta \|(\na^{3}n,\na^{4}v,\na^4q)\|_{L^2}^2.
\end{split}
\eeq
It is easy to deduce by using integrating by parts and Sobolev inequality that
\beq\label{F33-}
\int \na^{3}F_3\cdot \na^{3}q dx=-\int \na^{2}F_3\cdot \na^{4}q dx\leq C\|\na^2F_3\|_{L^2}\|\na^4q\|_{L^2}.
\eeq
With the aid of the definition of $F_3$, we have\beq\label{F33--}
\begin{split}
	\|\na^2F_3\|_{L^2}\leq& C\big(\|\na^2(v\cdot\na q)\|_{L^2}+\|\na^2(f(n+\bar\rho)\na^2 q)\|_{L^2}+\|\na^2(f(n+\bar\rho,q)\Psi(v))\|_{L^2}+\|\na^2\Psi(v)\|_{L^2}\\
	&\quad+\|\na^2(m(n+\bar\rho,q)\dive v)\|_{L^2}\big)\\
	\overset{def}{=}&L_1+L_2+L_3+L_4+L_5.
\end{split}
\eeq
According to Sobolev inequality, the commutator estimate in Lemma \ref{commutator} and the estimate \eqref{navna2v}, we obtain immediately\beno
L_1\leq  C\|v\|_{L^3}\|\na^3 q\|_{L^6}+C\|[\na^2,v]\cdot\na q\|_{L^2}\leq C\|v\|_{L^3}\|\na^4 q\|_{L^2}+C\|\na v\|_{L^3}\|\na^2q\|_{L^6}\leq C\delta\|\na^4 q\|_{L^2}.
\eeno
We employ Sobolev inequality and the estimate \eqref{na3na26}, to get\beno
\begin{split}
	L_2\leq& C\|(n+\bar\rho)\|_{L^\infty}\|\na^4 q\|_{L^2}+C\|\na^2 n\|_{L^6}\|\na^2q\|_{L^3}+C\|\na n\|_{L^\infty}\|\na n\|_{L^3}\|\na^2 q\|_{L^6}\\
	&\quad+C\|(1+|x|)^2\na^2\bar\rho\|_{L^\infty}\|\f{\na^2q}{(1+|x|)^2}\|_{L^2}+C\|\na n\|_{L^3}\|(1+|x|)\na\bar\rho\|_{L^\infty}\|\f{\na^2q}{1+|x|}\|_{L^6}\\
	&\quad+C\|(1+|x|)\na\bar\rho\|_{L^\infty}^2\|\f{\na^2q}{(1+|x|)^2}\|_{L^2}\\
	\leq& C\delta\|(\na^3 n,\na^4 q)\|_{L^2}.
\end{split}
\eeno
Applying the estiamte \eqref{navnav}, Sobolev and Hardy inequalities, we obtain\beno
\begin{split}
	L_3\leq& C\|(n+\bar\rho)\|_{L^\infty}\big(\|\na v\|_{L^3}\|\na^3v\|_{L^6}+\|\na^2 v\|_{L^3}\|\na^2v\|_{L^6}\big)+C\|\na^2 n\|_{L^6}\|\na v\|_{L^3}\|\na v\|_{L^\infty}\\
	&\quad+C\|(1+|x|)^2\na^2\bar\rho\|_{L^\infty}\|\na v\|_{L^3}\|\f{\na v}{(1+|x|)^2}\|_{L^6}+\|\na(n+\bar\rho)\|_{L^\infty}^2\|\na v\|_{L^3}\|\na v\|_{L^6}\\
	\leq& C\delta \|(\na^3 n,\na^4 v)\|_{L^2},
\end{split}
\eeno
where we have used the following estimate in the last inequality\beq\label{na2vna2v}
\|\na^2 v\|_{L^3}\|\na^2v\|_{L^6}\leq C\|\na^2 v\|_{L^2}^{\f34}\|\na^4v\|_{L^2}^{\f14}\| v\|_{L^2}^{\f14}\|\na^4v\|_{L^2}^{\f34},
\eeq
We then use the estimate \eqref{na2vna2v}, Sobolev and Hardy inequalities, to find\beno
\begin{split}
	L_4\leq C\|\na v\|_{L^3}\|\na^3v\|_{L^6}+C\|\na^2 v\|_{L^3}\|\na^2v\|_{L^6}\leq C\delta\|\na^4v\|_{L^2}.
\end{split}
\eeno
To deal with last term $L_5$, by virtue of the estimates \eqref{navna2v}-\eqref{navnav}, Sobolev and Hardy inequalities, we arrive at\beno
\begin{split}
	L_5\leq& C\|(n+\bar\rho+q)\|_{L^3}\|\na^3v\|_{L^6}+C\|\na^2(n+q)\|_{L^6}\|\na v\|_{L^3}+C\|\na(n+q)\|_{L^\infty}\|\na(n+q)\|_{L^6}\|\na v\|_{L^3}\\
	&\quad+C\|(1+|x|)^2\na^2\bar\rho\|_{L^3}\|\f{\na v}{(1+|x|)^2}\|_{L^6}+C\|na \bar\rho\|_{L^\infty}\|\na(n+q)\|_{L^6}\|{\na v}\|_{L^3}\\
	&\quad+C\|(1+|x|)\na \bar\rho\|_{L^\infty}\|(1+|x|)\na \bar\rho\|_{L^3}\|\f{\na v}{(1+|x|)^2}\|_{L^6}\\
	\leq&C\delta \|(\na^3 n,\na^4 v,\na^4q)\|_{L^2}.
\end{split}
\eeno
Hence, the combination of estimates of terms $L_1$ to $L_5$ and \eqref{F33--} implies directly\beno
\begin{split}
	\|\na^2F_3\|_{L^2}\leq& C\delta\|(\na^3 n,\na^4v,\na^4 q)\|_{L^2}.
\end{split}
\eeno
Inserting this estimate into \eqref{F33-}, we thereby deduce that
\beq\label{F33}
\begin{split}
	\int \na^{3}F_3\cdot \na^{3}q dx
	\leq C\delta \|(\na^{3}n,\na^{4}v,\na^4q)\|_{L^2}^2.
\end{split}
\eeq
Plugging the estimates \eqref{F13}, \eqref{F23} and \eqref{F33} into \eqref{ehk1} gives \eqref{en2} directly.
Therefore, the proof of this lemma is completed.
\end{proof}

Finally, we aim to recover the dissipation estimate for $n$.
\begin{lemm}\label{ennjc}
	Under the assumptions in Theorem \ref{them3}, for $k=1,2$, it holds that
	\beq\label{en3}
	\f{d}{dt}\int \na^k v\cdot\na^{k+1}ndx+\|\na^{k+1}n\|_{L^2}^2\leq C_1\|(\na^{k+1}v,\na^{k+2}v)\|_{L^2}^2,
	\eeq
	where $C_1$ is a positive constant independent of $t$.
\end{lemm}
\begin{proof}
	Applying differential operator $\na^k$ to $\eqref{ns5}_2$, multiplying the resulting equation by $\na^{k+1}n$, and integrating over $\mathbb{R}^3$, one arrives at\beq\label{vnjc}
	\begin{split}
		\int \na^{k}v_t\cdot\na^{k+1}n dx+\|\na^{k+1}n\|_{L^2}^2\leq C \|\na^{k+2}v\|_{L^2}^2+\int \na^kF_2\cdot\na^{k+1}n dx.
	\end{split}
	\eeq
	The way we deal with $\int \na^{k}v_t\cdot\na^{k+1}n dx$ is to turn the time derivative of the velocity to the density.
	Then, applying differential operator $\na^k$ to the mass equation $\eqref{ns5}_1$, we find
	\[\na^kn_t+\gamma\na^{k}\dive v=\na^k F_1.\]
	Hence, we can transform time derivative to the spatial derivative, i.e.,\beno
	\begin{split}
		\int \na^{k}v_t\cdot\na^{k+1}n dx=&\f{d}{dt}\int \na^{k}v\cdot\na^{k+1}n dx-\int \na^{k}v\cdot\na^{k+1}n_t dx\\
		=&\f{d}{dt}\int \na^{k}v\cdot\na^{k+1}n dx+\gamma\int \na^{k} v\cdot\na^{k+1}\dive v dx-\int \na^{k} v\cdot\na^{k+1}F_1 dx\\
		=&\f{d}{dt}\int \na^{k}v\cdot\na^{k+1}n dx-\gamma\|\na^{k}\dive v\|_{L^2}^2-\int \na^{k+1}\dive v\cdot\na^{k-1}F_1 dx
	\end{split}
	\eeno
	Substituting the identity above into \eqref{vnjc} and integrating by parts yield\beq\label{nvjc2}
	\begin{split}
		&\f{d}{dt}\int \na^{k}v\cdot\na^{k+1}n dx+\|\na^{k+1}n\|_{L^2}^2\\
		\leq& C \|(\na^{k+1}v,\na^{k+2}v)\|_{L^2}^2+C\int \na^{k+1}\dive v\cdot\na^{k-1}F_1dx-C\int \na^kF_2\cdot\na^{k+1}n dx.
	\end{split}
	\eeq
	As for the term of $F_1$, we have
	\beq\label{ss1}
	\begin{split}
		\Big|\int \na^{k+1}\dive v\cdot\na^{k-1}F_1 dx\Big|\leq C\|\na^{k+2}v\|_{L^2}\|\na^{k-1}F_1\|_{L^2}
		\leq C \delta \|\na^{k+1}n\|_{L^2}^2+C\|(\na^{k+1}v,\na^{k+2}v)\|_{L^2}^2,
	\end{split}
	\eeq
where $\|\na^{k-1}F_1\|_{L^2}$ ($k=1,2$) can be controlled in a similar way to the estimates of terms from \eqref{F11} and \eqref{F12} in Lemma \ref{enn-1}.
To deal with the term of $F_2$, we then derive in a similar way in \eqref{F2-estimate1} in Lemma \ref{enn}.
Hence, we give the estimate as follow
	\beq\label{ss2}
	\begin{split}
		\Big|\int \na^kF_2\cdot\na^{k+1}n dx\Big|\leq C \|\na^kF_2\|_{L^2}\|\na^{k+1}n\|_{L^2}\leq C \delta\|(\na^{k+1}n,\na^{k+2}v,\na^{k+2}q)\|_{L^2}^2.
	\end{split}
	\eeq
	We then utilize \eqref{ss1} and \eqref{ss2} in \eqref{nvjc2}, to deduce \eqref{en3} directly.
\end{proof}

\underline{\noindent\textbf{The proof of Proposition \ref{nenp}.}}
With the help of Lemmas \ref{enn-1}-\ref{ennjc},
it is easy to establish the estimate \eqref{eml}.
Therefore, we complete the proof of Proposition \ref{nenp}.

\subsection{Optimal decay of higher order derivative}
In this subsection, we will establish the optimal decay rate for the
second order spatial derivative of global solution.
In order to achieve this goal, the optimal decay rate of higher order spatial derivative will be established by the lower one.
In this aspect, developed by Schonbek(see \cite{Schonbek1985}), the Fourier splitting method is applied frequently to establish the optimal decay rate for higher order derivative of global solution in \cite{{Schonbek-Wiegner},{gao2016},{gao-wei-yao-D}}.
However, we are going to use time weighted energy estimate to solve this problem.

\begin{lemm}\label{N-1decay}
	Under the assumption of Theorem \ref{them3}, for $k=0,1,2$, it holds that
	\beq\label{n1h1}
	\|\na^k(n,v,q)\|_{H^{3-k}}\leq C (1+t)^{-\f34-\f k2},
	\eeq
	where $C$ is a positive constant independent of time.
\end{lemm}
\begin{proof}
Actually, the decay rate \eqref{basic-decay} implies \eqref{n1h1} holds true for the the case $k=0, 1$.
That is, the decay rate \eqref{n1h1}
holds on for the case $k=1$, i.e.,
	\beq\label{inducass}
	\|\na (n,v,q)\|_{H^{2}}\leq C (1+t)^{-\f54}.
	\eeq
	It remains the case of $k=2$ to be proven.
We take the integer $l=1$ in the estimate \eqref{eml} and
multiply it by $(1+t)^{\f52+\ep_0} (0<\ep_0<1)$, to discover
\beno
\begin{split}
\f{d}{dt}\Big\{(1+t)^{\f52+\ep_0} \mathcal{E}^3_1(t)\Big\}
+(1+t)^{\f52+\ep_0}\big(\|\na^{2}n\|_{H^{1}}^2
+\|\na^{2}(v,q)\|_{H^{2}}^2\big)\leq C(1+t)^{\f32+\ep_0} \mathcal{E}^3_1(t).
\end{split}
\eeno
Integrating with respect to $t$, using the equivalent relation \eqref{emleq} and the decay estimate \eqref{inducass}, one obtains
\beq\label{energy2}
\begin{split}
&(1+t)^{\f52+\ep_0} \mathcal{E}^3_1(t)	
    +\int_0^t(1+\tau)^{\f52+\ep_0}\big(\|\na^{2}n\|_{H^{1}}^2
    +\|\na^{2}(v,q)\|_{H^{2}}^2\big)d\tau\\
\leq& \mathcal{E}^3_1(0)+C\int_0^t(1+\tau)^{\f32+\ep_0} \mathcal{E}^3_1(\tau)d\tau\\
\leq&C\|\na(n_0,v_0,q_0)\|_{H^{2}}^2
     +C\int_0^t(1+\tau)^{\f32+\ep_0} \|\na(n,v,q)\|_{H^{2}}^2d\tau\\
\leq&C\|\na(n_0,v_0,q_0)\|_{H^{2}}^2
     +C\int_0^t(1+\tau)^{-1+\ep_0}d\tau\leq C(1+t)^{\ep_0}.
\end{split}
\eeq
On the other hand, taking $l=2$ in the estimate \eqref{eml}, we have
\beq\label{Ekk}
\begin{split}	
\f{d}{dt}\mathcal{E}^{3}_{2}(t)
+\|\na^{3}n\|_{L^2}^2+\|\na^{3}(v,q)\|_{H^{1}}^2\leq 0.
\end{split}
\eeq
We then multiply \eqref{Ekk} by $(1+t)^{\f52+m+\ep_0}$,
integrate over $[0, t]$ and use the estimate \eqref{energy2}, to find
\beno
\begin{split}
&(1+t)^{\f72+\ep_0}\mathcal{E}^{3}_{2}(t)
+\int_0^t(1+\tau)^{\f72+\ep_0}\big(\|\na^{3}n\|_{L^{2}}^2
+\|\na^{k+2}(v,q)\|_{H^{1}}^2\big)d\tau\\
\leq& \mathcal{E}^{3}_{2}(0)
+C\int_0^t(1+\tau)^{\f52+\ep_0}\mathcal{E}^{3}_{2}(\tau)d\tau\\
\leq&C\|\na^{2}(n_0,v_0,q_0)\|_{H^{1}}^2
+C\int_0^t(1+\tau)^{\f52+\ep_0}\|\na^{2}(n,v,q)\|_{H^{1}}^2d\tau
    	\leq C(1+t)^{\ep_0}.
\end{split}
\eeno
This, togeter with the equivalent relation \eqref{emleq}, yields immediately
\beno
\|\na^{2}(n,v,q)\|_{H^{1}}\leq C(1+t)^{-\f74}.
\eeno
Then, the decay estimate \eqref{n1h1} holds ture for case of $k=2$.
Therefore, we complete the proof of this lemma.
\end{proof}

\subsection{Optimal decay of critical derivative}

In this subsection, we aim to build the optimal decay rate for the third order spatial derivative of global solution $(n, v,q)$ as it tends to zero.
The decay rate of the third order derivative of global solution $(n, v,q)$ obtained in Lemma \ref{N-1decay}
is not optimal since it is same as that of the second one.
This is caused by the appearance of cross term
$\frac{d}{dt}\int \nabla^{2} v \cdot \nabla^3 n dx$ in energy
when we set up the dissipation estimate for the density in Lemma \ref{ennjc}.
Before giving the proof, we first introduce some notations that will be of frequency use in this subsection.
Let $0\leq\varphi_0(\xi)\leq1$ be a function in $C_0^{\infty}(\mathbb{R}^3)$ such that\begin{equation*}
\begin{split}
\varphi_0(\xi)=\left\{
\begin{array}{ll}
1,\quad \text{for}~~|\xi|\leq \f{\eta}{2},\\[1ex]
0,\quad\text{for}~~|\xi|\geq \eta,   \\[1ex]
\end{array}
\right.
\end{split}
\end{equation*}
where $\eta$ is a fixed positive constant, which will be chosen later. Based on the Fourier transform, we can define a low-medium-high-frequency decomposition $(f^l(x),f^h(x))$ for a function $f(x)$ as follows:
\beq\label{def-h-l}
f^l(x)\overset{def}{=}\mathcal{F}^{-1}(\varphi_0(\xi)\widehat{f}(\xi))~~\text{and}~~f^h(x)\overset{def}{=}f(x)-f^l(x).
\eeq

\begin{lemm}\label{highfrequency}
Under the assumptions of Theorem \ref{them3},
there exists a positive small constant $\eta_2$, such that
\beq\label{en6}
\begin{split}	&\f{d}{dt}\Big\{\|\na^{3}(n,v,q)\|_{L^2}^2-\eta_2\int_{|\xi|\geq\eta}\widehat{\na^{2}v}\cdot \overline{\widehat{\na^{3}n}}d\xi \Big\}+\|\na^{3}(v^h,q^h)\|_{L^2}^2+\eta_2\|\na^{3}n^h\|_{L^2}^2\\
		\leq&  C_2\|\na^{3}(n^l, v^l,q^l)\|_{L^2}^2+C(1+t)^{-6},
\end{split}
\eeq
where $C_2$ is a positive constant independent of time.
\end{lemm}

\begin{proof}
Taking differential operating $\na^{2}$ to the equation \eqref{ns5}, one obtains that
\beq\label{ns6}
\left\{\begin{array}{lr}
\na^{2}n_t +\gamma\na^{2}\dive v=\na^{2}F_1,\\
\na^{2}v_t+\gamma\na^3 n+\bar\lam \na^3 q-\mu_1\na^{2}\tri v-\mu_2\na^3\dive v =\na^{2}F_2,\\
\na^{2}q_t-\bar{\kappa}\na^{2}\tri q+\bar\lam \na^{2}\dive v=\na^{2}F_3.
\end{array}\right.
\eeq
We then take the Fourier transform of $\eqref{ns6}_2$, multiply the resulting equation by $\overline{\widehat{\na^{3}n}}$ and integrate on $\{\xi||\xi|\geq \eta\}$, to discover
\beq\label{f1}
\begin{split}
&\int_{|\xi|\geq\eta}\widehat{\na^{2}v_t}\cdot \overline{\widehat{\na^{3}n}}d\xi+\gamma\int_{|\xi|\geq\eta}|\widehat{\na^3n}|^2d\xi\\
=&\int_{|\xi|\geq\eta}\big(\mu_1\widehat{\na^{2}\tri v}+\mu_2\widehat{\na^3\dive v}\big)\cdot \overline{\widehat{\na^{3}n}}d\xi-\bar\lam\int_{|\xi|\geq\eta}\widehat{\na^{3}q}\cdot \overline{\widehat{\na^{3}n}}d\xi+\int_{|\xi|\geq\eta}\widehat{\na^{2}F_2}\cdot \overline{\widehat{\na^{3}n}}d\xi.
\end{split}
\eeq
It follows from $\eqref{ns6}_1$ that
\beno
\begin{split}
		\widehat{\na^{2}v_t}\cdot \overline{\widehat{\na^{3}n}}=&-i\xi\widehat{\na^{2}v_t}\cdot \overline{\widehat{\na^{2}n}}=-\widehat{\na^3v_t}\cdot \overline{\widehat{\na^{2}n}}\\
		=&-\pa_t(\widehat{\na^3v}\cdot \overline{\widehat{\na^{2}n}})+\widehat{\na^3v}\cdot \overline{\widehat{\na^{2}n_t}}\\
		=&-\pa_t(\widehat{\na^3v}\cdot \overline{\widehat{\na^{2}n}})-\gamma\widehat{\na^3v}\cdot
\overline{\widehat{\na^{2}\dive v}}+\widehat{\na^3v}\cdot \overline{\widehat{\na^{2}F_1}}.
	\end{split}
	\eeno
Then, we substitute this identity into identity \eqref{f1}, to find
\beq\label{f2}
\begin{split}
		&-\f{d}{dt}\int_{|\xi|\geq\eta}\widehat{\na^{3}v}\cdot \overline{\widehat{\na^{2}n}}d\xi+\gamma\int_{|\xi|\geq\eta}|\widehat{\na^3n}|^2d\xi\\
		=&\int_{|\xi|\geq\eta}\big(\mu_1\widehat{\na^{2}\tri v}+\mu_2\widehat{\na^3\dive v}\big)\cdot \overline{\widehat{\na^{3}n}}d\xi-\bar\lam\int_{|\xi|\geq\eta}\widehat{\na^{3}q}\cdot \overline{\widehat{\na^{3}n}}d\xi+\gamma\int_{|\xi|\geq\eta}\widehat{\na^{3}v}\cdot \overline{\widehat{\na^{2}\dive v}}d\xi \\
		&\quad-\int_{|\xi|\geq\eta}\widehat{\na^{3}v}\cdot \overline{\widehat{\na^{2}F_1}}d\xi +\int_{|\xi|\geq\eta}\widehat{\na^{2}F_2}\cdot \overline{\widehat{\na^{3}n}}d\xi\\
		\overset{def}{=}&M_1+M_2+M_3+M_4+M_5.
\end{split}
\eeq
The application of Cauchy inequality implies
\beq\label{N1}
\begin{split}
|M_1|
\leq C\int_{|\xi|\geq\eta}|\xi|^{7}|\widehat{v}||\widehat{n}|d\xi
\leq\ep \int_{|\xi|\geq\eta}|\xi|^{6}|\widehat{n}|^2d\xi
        +C_{\ep}\int_{|\xi|\geq\eta}|\xi|^{8}|\widehat{v}|^2d\xi,
\end{split}
\eeq
for some small constant $\ep$, which will be determined later.
It then follows from a similar way that\beq\label{N2}
\begin{split}
	|M_2|\leq C\int_{|\xi|\geq\eta}|\xi|^{6}|\widehat{q}||\widehat{n}|d\xi
	\leq\ep \int_{|\xi|\geq\eta}|\xi|^{6}|\widehat{n}|^2d\xi
	+C_{\ep}\int_{|\xi|\geq\eta}|\xi|^{6}|\widehat{q}|^2d\xi.
\end{split}
\eeq
Obviously, it holds true
\beq\label{N3}
|M_3| \leq  C\int_{|\xi|\geq\eta}|\xi|^{6}|\widehat{v}|^2d\xi.
\eeq
Using the Cauchy inequality and definition of $F_1$, one can show that
\beq\label{m3}
\begin{split}
|M_4|
\leq&C \int_{|\xi|\geq\eta}|\xi|^{8}|\widehat{v}|^2d\xi
       +C\int_{|\xi|\geq\eta}|\xi|^{2}|\widehat{F_1}|^2 d\xi \\
\leq&C \int_{|\xi|\geq\eta}|\xi|^{8}|\widehat{v}|^2d\xi
       +C\int_{|\xi|\geq\eta}|\xi|^{2}|\widehat{\na nv+n\na v}|^2d\xi+C\int_{|\xi|\geq\eta}|\xi|^{2}|\widehat{\na\bar\rho v+\bar\rho\na v}|^2d\xi\\
\overset{def}{=}&\int_{|\xi|\geq\eta}|\xi|^{8}|\widehat{v}|^2d\xi+M_{41}+M_{42}.
\end{split}
\eeq
The Plancherel Theorem  and Sobolev inequality yields directly
\beq\label{m31}
\begin{split}
M_{41}
\le  &C\|\na(\na nv+n\na v)\|_{L^2}^2\\
\leq &C\big(\|\na n\|_{L^\infty}^2\|\na v\|_{L^2}^2
      +\|\na^2 n\|_{L^2}^2\|v\|_{L^\infty}^2
      +\|n\|_{L^\infty}^2\|\na^{2}v\|_{L^2}^2\big)\\		
\leq&C\big(\|\na^{2}n\|_{H^1}^2\|\na v\|_{L^2}^2
      +\|\na(n,v)\|_{H^1}^2\|\na^{2}v\|_{L^2}^2\big)\\
\leq&C(1+t)^{-6},
\end{split}
\eeq
where we have used the decay \eqref{n1h1} in the last inequality.
We then apply Hardy inequality to obtain
\beq\label{m32}
\begin{split}
M_{42}
\leq &C\int_{|\xi|\geq\eta}|\xi|^{4}|\widehat{\na\bar\rho v+\bar\rho\na v}|^2d\xi
\le  C\|\na^{2}(\na \bar\rho v+\bar\rho\na v)\|_{L^2}^2 \\
		\leq&C\sum_{0\leq l\leq 2}\Big(\|(1+|x|)^{l+1}\na^{l+1}\bar\rho\|_{L^\infty}\|\f{\na^{2-l}v}{(1+|x|)^{l+1}}\|_{L^2}+\|(1+|x|)^{l}\na^{l}\bar\rho\|_{L^\infty}\|\f{\na^{3-l}v}{(1+|x|)^{l}}\|_{L^2}\Big)\|\na^{4}v\|_{L^2}\\
		\leq&C\delta\|(\na^{3}v,\na^{4}v)\|_{L^2}^2,
\end{split}
\eeq
where we have used the fact that for any suitable function $\phi$,
there exists a positive constant $C$ depending only on $\eta$ such that
\beno
\int_{|\xi|\geq\eta}|\xi|^{2}|\widehat{\phi}|^2d\xi\leq C\int_{|\xi|\geq\eta}|\xi|^{4}|\widehat{\phi}|^2d\xi.
\eeno
Substituting the estimates \eqref{m31} and \eqref{m32} into \eqref{m3},
it is easy to check that
\beq\label{N4}
|M_4|\leq C\delta\|(\na^{3}v,\na^{4}v)\|_{L^2}^2 +C(1+t)^{-6}.
\eeq
Applying the definition of $F_2$ and Cauchy inequality, one can get that
\beq\label{m4}
\begin{split}
|M_5|
\leq &C\int_{|\xi|\geq\eta}|\xi|^{5}|\widehat{F_2}|| {\widehat{n}}|d\xi\\
		\leq&C\int_{|\xi|\geq\eta}|\xi|^{5}| {\widehat{n}}||\widehat{v\cdot \na v}|d\xi
     +C\int_{|\xi|\geq\eta}|\xi|^{5}| {\widehat{n}}
        ||\widehat{n(\tri v+\na\dive v)}|d\xi+C\int_{|\xi|\geq\eta}|\xi|^{5}| {\widehat{n}}|
          |\widehat{(n+q)\na (n+q)}|d\xi\\
          &
      +C\int_{|\xi|\geq\eta}|\xi|^{5}| {\widehat{n}}|
          |\widehat{\bar\rho(\tri v+\na\dive v)}|d\xi+C\int_{|\xi|\geq\eta}|\xi|^{5}| {n}|
          |\widehat{{\bar\rho}\na(n+q)}|d\xi+C\int_{|\xi|\geq\eta}|\xi|^{5}| {\widehat{n}}|
          |\widehat{(n+q)\na\bar\rho}|d\xi\\
\overset{def}{=}&M_{51}+M_{52}+M_{53}+M_{54}+M_{55}+M_{56}.
\end{split}
\eeq
By virtue of Plancherel Theorem, Sobolev inequality,
commutator estimate in Lemma \ref{commutator}, and the estimate \eqref{n1h1}, we obtain
\beq\label{M41}
\begin{split}
M_{51}
\leq &\ep \|\na^3n\|_{L^2}^2+C_{\ep}\|\na^{2}(v\cdot\na v)\|_{L^2}^2\\
\leq &\ep \|\na^3n\|_{L^2}^2+C_{\ep}\|v\|_{L^\infty}^2\|\na^3v\|_{L^2}^2
      +C_{\ep}\|[\na^{2},v]\cdot\na v\|_{L^2}^2\\
\leq &\ep \|\na^3n\|_{L^2}^2+C_{\ep}\|\na v\|_{H^1}^2\|\na^3v\|_{L^2}^2
      +C_{\ep}\|\na v\|_{L^\infty}^2\|\na^{2} v\|_{L^2}^2\\
\leq &\ep \|\na^3n\|_{L^2}^2+C_{\ep}(1+t)^{-6}.
\end{split}
\eeq
Similarly, it also holds that
\beq\label{m42}
\begin{split}
M_{52}
\leq &\ep \|\na^3n\|_{L^2}^2+C_{\ep}\|\na^{2}\big(n
        (\tri v+\na\dive v)\big)\|_{L^2}^2\\
\leq &\ep\|\na^3n\|_{L^2}^2+C_{\ep}\|n\|_{L^\infty}^2
      \|\na^{4}v\|_{L^2}^2
      +C_{\ep}\|\na^2n\|_{L^3}^2 \|\na^2v\|_{L^6}^2\\
\leq &\ep \|\na^3n\|_{L^2}^2+C_{\ep}\delta\|\na^{4}v\|_{L^2}^2
      +C_{\ep}(1+t)^{-7}.
\end{split}
\eeq
One can deal with the term $M_{53}$ in the manner of $M_{52}$. It holds true
\beq\label{m43}
\begin{split}
	M_{53}
	\leq&\ep \|\na^3n\|_{L^2}^2+C_{\ep}\|(n,q)\|_{L^\infty}^2\|\na^{3}(n,q)\|_{L^2}^2+C_{\ep}\|\na(n,q)\|_{L^3}^2\|\na^{2}(n,q)\|_{L^6}^2\\
	\leq&(\ep+C_\ep\delta) \|\na^3(n,q)\|_{L^2}^2.
\end{split}
\eeq
As for $M_{54}$, thanks to H\"older and Hardy inequalities, we find
\beq\label{m44}
\begin{split}
	M_{54}
	\leq&\ep \|\na^3n\|_{L^2}^2+C_{\ep}\|\bar\rho\|_{L^\infty}^2\|\na^{4}v\|_{L^2}^2+C_{\ep}\|\na^2\bar\rho\|_{L^3}^2\|\na^{2}v\|_{L^6}^2\\
	\leq&\ep \|\na^3n\|_{L^2}^2+C_\ep\delta\|(\na^3v,\na^4v)\|_{L^2}^2.
\end{split}
\eeq
Finally, let us deal with the term $M_{55}$ and $M_{56}$ together.
Indeed, the Hardy inequality yields directly
\beno
\begin{split}
M_{55}+M_{56}		\leq&\ep\|\na^{3}n\|_{L^2}^2+C_{\ep}\|\bar\rho\|_{L^\infty}^2\|\na^3(n,q)\|_{L^2}^2+C_{\ep}\|(1+|x|)\na\bar\rho\|_{L^\infty}^2\|\f{\na^2(n,q)}{1+|x|}\|_{L^2}^2\\
	&\quad+C_{\ep}\|(1+|x|)^2\na^2\bar\rho\|_{L^\infty}^2\|\f{\na(n,q)}{(1+|x|)^2}\|_{L^2}^2+C_{\ep}\|(1+|x|)^3\na^3\bar\rho\|_{L^\infty}^2\|\f{(n,q)}{(1+|x|)^3}\|_{L^2}^2\\
	\leq&(\ep+C_{\ep}\d)\|\na^{3}n\|_{L^2}^2+C_{\ep}\d\|\na^{3}q\|_{L^2}^2.
\end{split}
\eeno
This bound, together with estimates \eqref{m4}-\eqref{m44}, leads us to get
\beq\label{N5}
|M_{5}|\leq (\ep+C_{\ep}\d)\|\na^{3}n\|_{L^2}^2+C_{\ep}\d\|(\na^3 v,\na^4 v,\na^{3}q)\|_{L^2}^2+C_{\ep}(1+t)^{-6}.
\eeq
Substituting the estimates \eqref{N1}-\eqref{N3}, \eqref{N4} and \eqref{N5} into \eqref{f2}, we find
\beno
\begin{split}
&-\f {d}{dt}\int_{|\xi|\geq\eta}\widehat{\na^{3}v}\cdot \overline{\widehat{\na^{2}n}}d\xi+\gamma\int_{|\xi|\geq\eta}|\widehat{\na^3n}|^2d\xi\\
&\leq(\ep+C_{\ep}\delta)\|\na^3 n\|_{L^2}^2+C_{\ep}\d\|(\na^3 v,\na^4 v,\na^{3}q)\|_{L^2}^2+C_{\ep}(1+t)^{-6}.
	\end{split}
	\eeno
Recalling the definition \eqref{def-h-l}, there exists a positive constant $C$ such that
	\beq\label{vhvl}
	\begin{split}
		\|\na^3 v^h\|_{L^2}^2\leq C\|\na^{4}v^h\|_{L^2}^2,\quad \|\na^{4} v^l\|_{L^2}^2\leq C \|\na^{3}v^l\|_{L^2}^2,
	\end{split}
	\eeq
	and choosing $\ep$ and $\delta$ suitably small, we deduce that
	\beq\label{en5}
	\begin{split}
		&-\f{d}{dt}\int_{|\xi|\geq\eta}\widehat{\na^{3}v}\cdot \overline{\widehat{\na^{2}n}}d\xi+\gamma\int_{|\xi|\geq\eta}|\widehat{\na^3n}|^2d\xi
		\leq C\|\na^3 (n^l,v^l,q^l)\|_{L^2}^2+C_{3} \|\na^{4}v^h\|_{L^2}^2+C(1+t)^{-6}.
	\end{split}
	\eeq
Recalling the estimate \eqref{en2} in Lemma \ref{enn}, the following estimate holds ture
\beq\label{en4}
\f{d}{dt}\|\na^{3}(n,v,q)\|_{L^2}^2+\|\na^{4}(v,q)\|_{L^2}^2\leq C\delta \|(\na^{3}n,\na^{4}v,\na^4q)\|_{L^2}^2.
\eeq
We multiply \eqref{en5} by $\eta_2$, then add to \eqref{en4}, and choose $\delta$ and $\eta_2$ suitably small, to discover
\begin{equation*}
\begin{aligned}		\f{d}{dt}\Big\{\|\na^{3}(n,v,q)\|_{L^2}^2-\eta_2\int_{|\xi|\geq\eta}\widehat{\na^{2}v}\cdot \overline{\widehat{\na^{3}n}}d\xi \Big\}
+\|\na^{4}(v,q)\|_{L^2}^2+\eta_2\|\na^{3}n^h\|_{L^2}^2
\leq C_2\|\na^{3}(n^l,v^l,q^l)\|_{L^2}^2+C(1+t)^{-6}.
\end{aligned}
\end{equation*}
Using \eqref{vhvl} once again, we obtain that
\begin{equation*}
\begin{aligned}	
\f{d}{dt}\Big\{\|\na^{3}(n,v,q)\|_{L^2}^2
  -\eta_2\int_{|\xi|\geq\eta}\widehat{\na^{2}v}\cdot
  \overline{\widehat{\na^{3}n}}d\xi \Big\}
  +\|\na^{3}(v^h,q^h)\|_{L^2}^2+\eta_2\|\na^{3}n^h\|_{L^2}^2
\leq C_2\|\na^{3}(n^l, v^l,q^l)\|_{L^2}^2+C(1+t)^{-6}.
\end{aligned}
\end{equation*}
Thus, the proof of this lemma is completed.
\end{proof}

It is noted that the low frequency of $\na^3(n,v,q)$ in the right handside of the estimate \eqref{en6} in Lemma \ref{highfrequency} need to be handled.
For this purpose, we first analyze the initial value problem for the linearized system of \eqref{ns5}:
\beq\label{linear}
\left\{\begin{array}{lr}
\widetilde n_t +\gamma\dive\widetilde v=0,\quad (t,x)\in \mathbb{R}^{+}\times \mathbb{R}^3,\\
\widetilde u_t+\gamma\na \widetilde n+\bar\lam\na^3\widetilde q-\mu_1\tri \widetilde v-\mu_2\na\dive \widetilde v =0,\quad (t,x)\in \mathbb{R}^{+}\times \mathbb{R}^3,\\
\widetilde q_t-\bar{\kappa}\tri \widetilde q+\bar\lam \dive \widetilde v=0,\quad (t,x)\in \mathbb{R}^{+}\times \mathbb{R}^3,\\
(\widetilde n,\widetilde v,\widetilde q)|_{t=0}=(n_0,v_0,q_0),\quad x\in \mathbb{R}^3.
\end{array}\right.
\eeq
    In terms of the semigroup theory for evolutionary equations, one can represent the solution $(\widetilde n,\widetilde v,\widetilde q)$ of the linearized system \eqref{linear} as follows:
    \beq\label{U}
    \left\{\begin{array}{lr}
    \widetilde U_t=A\widetilde U,\quad t\geq0,\\
    \widetilde U(0)=U_0,
\end{array}\right.\eeq
    where $\widetilde U \overset{def}{=}(\widetilde{n},\widetilde{v},\widetilde q)^t$,
    $U_0 \overset{def}=(n_0,v_0,q_0)^t$
    and the matrix-valued differential operator $A$ is given by
    \beno
    A = {\left(
	\begin{matrix}
		0 & -\gamma \dive &0\\
		-\gamma \na & \mu_1\tri+\mu_2\na\dive&-\bar\lam\na\\
		0& -\bar\lam \dive& \bar\kappa\tri
	\end{matrix}
	\right).}
    \eeno
    We then denote $S(t)\overset{def} =e^{tA}$, and recall the system \eqref{U}, to find
    \beq\label{uexpress}\widetilde U(t)=S(t)U_0=e^{tA} U_0,\quad t\geq0. \eeq
    Then, it is easy to deduce that the following estimate holds
    \beq\label{linearlow}
    \|\na^3(S(t)U_0)\|_{L^2}\leq C(1+t)^{-\f94}\|U_0\|_{L^1\cap H^N},
    \eeq
    where $C$ is a positive constant independent of time.
    The proof of the estimate \eqref{linearlow} can be found in \cite{chen2021,duan2007}, so we omit here.
    Finally, let us denote $F(t)=(F_1(t),F_2(t),F_3(t))^{t}$, then
    the system \eqref{ns5} can be rewritten as follows:
    \beq\label{nonlinear}
    \left\{\begin{array}{lr}
    	U_t=A U+F,\\
    	U(0)=U_0.
    \end{array}\right.
    \eeq
    In term of the semigroup and Duhamel's principle, the solution of system \eqref{ns5} can be expressed as
    \beq\label{Uexpress}
    U(t)=S(t)U_0+\int_0^tS(t-\tau)F(\tau)d\tau.
    \eeq
    Now, one can establish the following estimate for the low frequency of $\na^3(n,v,q)$ as follows:
\begin{lemm}\label{lowfrequency}
Under the assumption of Theorem \ref{them3}, it holds that
\beq\label{lowfre}
\|\na^3(n^l, v^l,q^l)(t)\|_{L^2}
\leq C \delta \sup_{0\leq s\leq t}\|\na^3(n,v,q)(s)\|_{L^2}+C(1+t)^{-\f94},
\eeq
where $C$ is a positive constant independent of time.
\end{lemm}
\begin{proof}
The formula \eqref{Uexpress} yields directly
\beno
   \na^3(n, v,q)=\na^3[S (t)U_0]+\int_0^t\na^{3}[S(t-\tau)F(\tau)]d\tau.
\eeno
This implies that
\beq\label{nvexpress}
	\|\na^3(n^l, v^l,q^l)\|_{L^2}
    \leq \|\na^3(S(t)U_0)^l\|_{L^2} +\int_0^t\|\na^{3}[S(t-\tau)F(\tau)]^l\|_{L^2}d\tau.
\eeq
Since the initial data $U_0=(n_0,v_0,q_0)\in L^1\cap H^3$, it follows from the estimate \eqref{linearlow} that
\beq\label{U0es}
	\|\na^3[S (t)U_0]^l\|_{L^2}\leq C(1+t)^{-\f94}\|U_0\|_{L^1\cap H^3}.
\eeq
We then apply Sobolev inequality to obtain that\beq\label{nvl}
\begin{split}
	&\int_0^t\|\na^{3}[S(t-\tau)F(\tau)]^l\|_{L^2}d\tau
	\leq\int_0^t\||\xi|^{3}|\widehat{S}(t-\tau)||\widehat F(\tau)|\|_{L^2(|\xi|\leq \eta)}d\tau\\
	\leq&\int_0^{\f t2}\||\xi|^{3}|\widehat{S}(t-\tau)|\|_{L^2(|\xi|\leq \eta)}\|\widehat F(\tau)\|_{L^\infty(|\xi|\leq \eta)}d\tau+\int_{\f t2}^t\||\xi||\widehat{S}(t-\tau)|\|_{L^2(|\xi|\leq \eta)}\||\xi|^{2}\widehat F(\tau)\|_{L^\infty(|\xi|\leq \eta)}d\tau\\
	\leq&\int_0^{\f t2}(1+t-\tau)^{-\f94}\|\widehat F(\tau)\|_{L^\infty(|\xi|\leq \eta)}d\tau+\int_{\f t2}^t(1+t-\tau)^{-\f54}\||\xi|^{2}\widehat F(\tau)\|_{L^\infty(|\xi|\leq \eta)}d\tau\\
	\overset{def}{=}&N_1+N_2.
\end{split}
\eeq
Now the first term on the right handside of \eqref{nvl} can be estimated as follows:
\beq\label{ss1N}
\begin{split}
N_1=
\int_0^{\f t2}(1+t-\tau)^{-\f94}\|F\|_{L^1}d\tau
\leq C\int_0^{\f t2}(1+t-\tau)^{-\f94} \big(\|F_1\|_{L^1}+\|F_2\|_{L^1}+\|F_3\|_{L^1}\big)d\tau.
\end{split}
\eeq
We compute by the definitions of $F_i(i=1,2,3)$ and decay estimate \eqref{n1h1} that
\beno\begin{split}
\|F_1\|_{L^1}\leq& C\|\na (n,v)\|_{L^2}\|(n,v)\|_{L^2}+C\|(1+|x|)\na \bar\rho\|_{L^2}\|\f{v}{1+|x|}\|_{L^2}+C\|\bar\rho\|_{L^2}\|\na v\|_{L^2}\leq C\delta (1+t)^{-\f54},\\
\|F_2\|_{L^1}\leq& C\|v\|_{L^2}\|\na v\|_{L^2}+C\|n+\bar\rho\|_{L^2}\|\na^2 v\|_{L^2}+C\|n+q\|_{L^2}\|\na (n+q)\|_{L^2}\\
&\quad+C\|\f{(n+q)}{1+|x|}\|_{L^2}\|(1+|x|)\na \bar\rho\|_{L^2}+C\|\na(n+q)\|_{L^2}\|\bar\rho\|_{L^2}\\
\leq&C\delta (1+t)^{-\f54},\\
\|F_3\|_{L^1}\leq& C\|v\|_{L^2}\|\na q\|_{L^2}+C\|(n+\bar\rho)\|_{L^2}\|\na^2 q\|_{L^2}+C\|(n+q)\|_{L^2}\|\na v\|_{L^2}+C(\|n\|_{L^\infty}+1)\|\na v\|_{L^2}\|\na v\|_{L^2}\\
&\quad+C\|\bar\rho\|_{L^2}\|\na v\|_{L^2}+C\|\bar\rho\|_{L^\infty}\|\na v\|_{L^2}\|\na v\|_{L^2}\\
\leq&C\delta (1+t)^{-\f54}.
\end{split}
\eeno
Substituting three estimates above into \eqref{ss1N},
and using the estimate in Lemma \ref{tt2}, it holds that
\beq\label{F1}
	\begin{split}
		N_1\leq C \int_0^{\f t2}(1+t-\tau)^{-\f94}(1+\tau)^{-\f54}d\tau\leq C (1+t)^{-\f94}.
\end{split}
\eeq
Next, let us deal with the $N_2$ term. It follows directly
\beq\label{L}
\begin{split}
	\||\xi|^{2}\widehat F\|_{L^\infty(|\xi|\leq \eta)}
	\leq C \|[\na^{2}F_1]^l\|_{L^1}+C\||\xi|^{2}\widehat F_2\|_{L^\infty(|\xi|\leq \eta)}+C\||\xi|^{2}\widehat F_3\|_{L^\infty(|\xi|\leq \eta)}.
\end{split}
\eeq
First of all, applying the decay estimate \eqref{n1h1} and Hardy inequality, then the first term in the right handside of \eqref{L} can be estimated as follows
\beq\label{L1}
\begin{split}
	\|[\na^{2}F_1]^l\|_{L^1}
	\leq&C\|[\na^{2}(\na nv+n\na v)]^l\|_{L^1}+C\|[\na^{2}(\na \bar\rho v+\bar\rho\na v)]^l\|_{L^1}\\
	\leq& C\sum_{l=0}^{2}\big(\|\na^{l+1}n\|_{L^2}\|\na^{2-l}v\|_{L^2}+\|\na^{l}n\|_{L^2}\|\na^{3-l}v\|_{L^2}\\
	&\quad+C\|(1+|x|)^{l+1}\na^{l+1}\bar\rho\|_{L^2}\|\f{\na^{2-1}v}{(1+|x|)^{l+1}}\|_{L^2}+\|(1+|x|)^{l}\na^{l}\bar\rho\|_{L^2}\|\f{\na^{3-l}v}{(1+|x|)^{l}}\|_{L^2} \Big)\\
	\leq& C (1+t)^{-\f52}+C \delta\|\na^3 v\|_{L^2}.
\end{split}
\eeq
For any smooth function $\phi$, there exists a positive constant $C$
depending only on $\eta$, such that
$$\||\xi|^{2}\widehat \phi\|_{L^\infty(|\xi|\leq \eta)}\leq C\||\xi|\widehat \phi\|_{L^\infty(|\xi|\leq \eta)},$$
by virtue of the decay estimate \eqref{n1h1} and Hardy inequality, then we find that \beq\label{L222}
\begin{split}
	\||\xi|^{2}\widehat F_2\|_{L^\infty(|\xi|\leq \eta)}\leq&C\|[\na^{2}\big(v\cdot\na v+(n+q)\na(n+q)+\bar\rho\na(n+q)+(n+q)\na\bar\rho\big)]^l\|_{L^1}\\
	&\quad+C\|[\na\big((n+\bar\rho)(\tri v+\na\dive v)\big)]^l\|_{L^1} \\
	\leq&\sum_{l=0}^{2}\big(\|\na^{l}v\|_{L^2}\|\na^{3-l}v\|_{L^2}+\|\na^{l}(n+q)\|_{L^2}\|\na^{3-l}(n+q)\|_{L^2}\\
	&\quad+\|(1+|x|)^{l}\na^{l}\bar\rho\|_{L^2}\|\f{\na^{3-l}(n+q)}{(1+|x|)^{l}}\|_{L^2}+\|(1+|x|)^{l+1}\na^{l+1}\bar\rho\|_{L^2}\|\f{\na^{2-l}(n+q)}{(1+|x|)^{l+1}}\|_{L^2} \big)\\
	&\quad+\sum_{l=0,1}\big( \|\na^{l}n\|_{L^2}\|\na^{3-l}v\|_{L^2}+C\|(1+|x|)^{l}\na^{l}\bar\rho\|_{L^2}\|\f{\na^{3-l}v}{(1+|x|)^{l}}\|_{L^2}\big)\\
	\leq&C(1+t)^{-\f52}+C \delta\|\na^3 (n,v,q)\|_{L^2}.
\end{split}
\eeq
In view of the decay estimate \eqref{n1h1} and Hardy inequality, we also have
\beq\label{LL3}
\begin{split}
		\||\xi|^{2}\widehat F_3\|_{L^\infty(|\xi|\leq \eta)}\leq&C\|[\na^{2}\big(v\cdot\na q+(n+q)\dive v+n\Psi(v)+\Psi(v)+\bar\rho\dive v+\bar\rho\Psi(v)\big)]^l\|_{L^1}\\
		&\quad+C\|[\na\big((n+\bar\rho)\tri q\big)]^l\|_{L^1}\\
		\leq& C\sum_{l=0}^{2}\big(\|\na^{l}v\|_{L^2}\|\na^{3-l}q\|_{L^2}+\|\na^{l}(n+q)\|_{L^2}\|\na^{3-l}v\|_{L^2}+\|\na^{l+1}v\|_{L^2}\|\na^{3-l}v\|_{L^2}\\
		&\quad+\|\na^{l}n\|_{L^2}\sum_{j=0}^{2-l}\|\na^{j+1} v\|_{L^3}\|\na^{3-l-j} v\|_{L^6}+\|(1+|x|)^{l}\na^{l}\bar\rho\|_{L^2}\|\f{\na^{3-l}v}{(1+|x|)^{l}}\|_{L^2}\\
		&\quad+\|(1+|x|)^{l}\na^{l}\bar\rho\|_{L^\infty}\sum_{j=0}^{2-l}\|\na^{j+1} v\|_{L^2}\|\f{\na^{3-l-j} v}{(1+|x|)^{l}}\|_{L^2}\big)\\
		&\quad+\sum_{l=0,1}\big( \|\na^{l}n\|_{L^2}\|\na^{3-l}q\|_{L^2}+C\|(1+|x|)^{l}\na^{l}\bar\rho\|_{L^2}\|\f{\na^{3-l}q}{(1+|x|)^{l}}\|_{L^2}\big)\\
		\leq&C (1+t)^{-\f52}+C \delta\|\na^3 (v,q)\|_{L^2}.
	\end{split}
	\eeq
We then conclude from \eqref{L}-\eqref{LL3} that\beno
\begin{split}
	\||\xi|^{2}\widehat F\|_{L^\infty(|\xi|\leq \eta)}\leq&C\delta\|\na^3 (n,v,q)\|_{L^2}+ C(1+t)^{-\f52},
\end{split}
\eeno
which, together with the definition of term $N_2$ and the estimate in Lemma \ref{tt2}, yields directly
\beq\label{kF1}
\begin{split}
	N_2
	\leq&C \int_{\f t2}^t(1+t-\tau)^{-\f54}\big(\delta\|\na^3 (n,v,q)\|_{L^2}+ (1+\tau)^{-\f52}\big)d\tau\\
	\leq&C\delta\sup_{0\leq\tau\leq t}\|\na^3 (n,v,q)\|_{L^2}\int_{\f t2}^t(1+t-\tau)^{-\f54}d\tau+C(1+t)^{-\f52}\\
	\leq&C\delta\sup_{0\leq\tau\leq t}\|\na^3 (n,v,q)\|_{L^2}+C(1+t)^{-\f52}.
\end{split}\eeq
Substituting \eqref{F1} and \eqref{kF1} into \eqref{nvl}, one arrives at
\beq\label{Unon}
    \int_0^t\|\na^{3}[S(t-\tau)F(U(\tau))]^l\|_{L^2}d\tau
    \leq C\delta\sup_{0\leq\tau\leq t}\|\na^3 (n,v,q)\|_{L^2}+C(1+t)^{-\f94}.
\eeq
Inserting \eqref{U0es} and \eqref{Unon} into \eqref{nvexpress}, one obtains immediately that
\beno
	\|\na^3(n^l,v^l,q^l)\|_{L^2}\leq C \delta \sup_{0\leq s\leq t}\|\na^3(n,v,q)\|_{L^2}+C(1+t)^{-\f94}.
\eeno
Thus, we finish the proof of this lemma.
\end{proof}

Finally, we aim to establish optimal decay rate for the third order spatial derivative of the solution.
\begin{lemm}\label{optimaln}
	Under the assumption of Theorem \ref{them3}, it holds that
	\beq\label{n1h2}
	\|\na^3(n,v,q)(t)\|_{L^2}\leq C (1+t)^{-\f94},
	\eeq
	where $C$ is a positive constant independent of time.
\end{lemm}
\begin{proof}
We first rewrite the estimate \eqref{en6} in Lemma \ref{highfrequency} as\beq\label{ddt}
	\f{d}{dt}\widetilde{\mathcal E}^3(t)+\|\na^{3}(v^h,q^h)\|_{L^2}^2+\eta_2\|\na^{3}n^h\|_{L^2}^2
	\leq  C_2\|\na^{3}(n,v,q)^l\|_{L^2}^2+C(1+t)^{-6}.
\eeq
Here the energy $\widetilde{\mathcal E}^3(t)$ is defined by
\[\widetilde{\mathcal E}^3(t)\overset{def}{=}\|\na^{3}(n,v,q)\|_{L^2}^2-\eta_2\int_{|\xi|\geq\eta}\widehat{\na^{2}v}\cdot \overline{\widehat{\na^{3}n}}d\xi.\]
Thanks to Young inequality, by choosing $\eta_2$ small enough, we obtain the following equivalent relation\beq\label{endj}
	c_3\|\na^{3}(n,v)\|_{L^2}^2\leq\widetilde{\mathcal E}^3(t)\leq c_4 \|\na^{3}(n,v)\|_{L^2}^2,
\eeq
where the constants $c_3$ and $c_4$ are independent of time.
We then add on both sides of \eqref{ddt} by $\|\na^{3}(n^l,v^l,q^l)\|_{L^2}^2$ and apply the estimate \eqref{lowfre} in Lemma \ref{lowfrequency}, to discover
\beno
\begin{split}
	\f{d}{dt}\widetilde{\mathcal E}^3(t)+\|\na^{3}(n,v,q)\|_{L^2}^2
	\leq ( C_2+1)\|\na^{3}(n^l,v^l,q^l)\|_{L^2}^2+C(1+t)^{-6}\leq C\delta \sup_{0\leq \tau\leq t}\|\na^3(n,v,q)\|_{L^2}^2+C(1+t)^{-\f92}.
\end{split}
\eeno
In view of the equivalent relation \eqref{endj}, we have\beq\label{en7}
\begin{split}
	\f{d}{dt}\widetilde{\mathcal E}^3(t)+\widetilde{\mathcal E}^3(t)
	\leq C\delta \sup_{0\leq \tau\leq t}\|\na^3(n,v,q)\|_{L^2}^2+C(1+t)^{-\f92}.
\end{split}
\eeq
This, together with Gronwall inequality, gives immediately\beq\label{estimateE}
\begin{split}
	\widetilde{\mathcal E}^3(t)\leq e^{-t} \widetilde{\mathcal E}^3(0)+C\delta\sup_{0\leq \tau\leq t}\|\na^3(n,v,q)\|_{L^2}^2\int_0^te^{\tau-t}d\tau+C\int_0^te^{\tau-t}(1+\tau)^{-\f92}d\tau.
\end{split}
\eeq
By some direct calculations, we can deduce easily
    $$
	\int_0^te^{\tau-t}d\tau\leq C \quad\text{and}\quad
	\int_0^te^{\tau-t}(1+\tau)^{-\f92}d\tau\leq C(1+t)^{-\f92}.$$
The equivalent relation \eqref{endj} and \eqref{estimateE} gives immediately
\beno
\begin{split}
	\sup_{0\leq \tau\leq t}\|\na^3(n,v,q)(\tau)\|_{L^2}^2\leq Ce^{-t}\|\na^3(n_0,v_0,q_0)\|_{L^2}^2+C\delta\sup_{0\leq \tau\leq t}\|\na^3(n,v,q)\|_{L^2}^2+ C(1+t)^{-\f92}.
\end{split}
\eeno
By applying the smallness of $\delta$, we have
\beno
	\sup_{0\leq \tau\leq t}\|\na^3(n,v,q)(\tau)\|_{L^2}^2\leq C(1+t)^{-\f92}.
\eeno
Consequently, this completes the proof of this lemma.
\end{proof}

\underline{\noindent\textbf{The Proof of Theorem  \ref{them3}.}}
Combining the estimate \eqref{n1h1} in Lemma \ref{N-1decay} with
estimate \eqref{n1h2} in Lemma \ref{optimaln}, we then can obtain the
decay rate \eqref{kdecay} in Theorem \ref{them3}.
Therefore, we complete the proof of Theorem \ref{them3}.






\section*{Acknowledgements}

This research was partially supported by NNSF of China (11801586, 11971496, 12026244), Guangzhou Science and technology project of China (202102020769),
National Key Research and Development Program of China (2020YFA0712500).

\phantomsection

\end{document}